\documentclass[a4paper]{amsart}
\usepackage{a4wide}
\usepackage{times,epsfig}
\usepackage{bbm}
\usepackage{amsfonts,amscd,amsmath,amssymb,amsthm}
\usepackage{mathtools}

\DeclareMathAlphabet{\Ma}{U}{msa}{m}{n}
\DeclareMathAlphabet{\Mb}{U}{msb}{m}{n}

\DeclareMathAlphabet{\Meuf}{U}{euf}{m}{n}
\DeclareSymbolFont{ASMa}{U}{msa}{m}{n}
\DeclareSymbolFont{ASMb}{U}{msb}{m}{n}
\DeclareMathOperator{\ran}{ran}

\newcommand{\scalar}[2]{\langle#1\,,#2\rangle}
\newcommand{\scalarb}[2]{\langle#1\,,#2\rangle_{\partial \Omega}}
\newcommand{\pair}[2]{(#1\,,#2)}
\newcommand{\pairb}[2]{(#1\,,#2)_{\partial \Omega}}
\newcommand{\norm}[1]{\|#1\|}
\newcommand{\normm}[1]{|\negthinspace\|#1\|\negthinspace|}
\newcommand{\vol}{\mathrm{vol}}
\renewcommand{\H}{\mathcal{H}}

\newcommand{\D}{\mathcal{D}}
\renewcommand{\d}{\mathrm{d}}
\newcommand{\pO}{{\partial \Omega}}

\newcommand{\C}{\mathcal{C}}
\newcommand{\diff}{\mathrm{d}}

\newcommand{\1}{\mathbb{I}}
\newcommand{\bx}{\mathbf{x}}
\newcommand{\bt}{\boldsymbol\theta}

\newtheorem{theorem}{Theorem}

\newtheorem{corollary}[theorem]{Corollary}
\newtheorem{proposition}[theorem]{Proposition}
\newtheorem{definition}[theorem]{Definition}
\newtheorem{lemma}[theorem]{Lemma}
\newtheorem{example}[theorem]{Example}
\newtheorem{remark}[theorem]{Remark}

\numberwithin{equation}{section}
\numberwithin{theorem}{section}

\title[Self-adjoint Extensions of the Laplace-Beltrami operator]{
Self-adjoint extensions of the Laplace-Beltrami operator
and unitaries at the boundary
}

\date{}
\author{Alberto Ibort}
\author{Fernando Lled\'{o}}
\author{Juan Manuel P\'erez-Pardo}
\address{Department of Mathematics,
University Carlos~III, Madrid, Avda. de la Universidad 30, E-28911 Legan\'es
(Madrid), Spain
and
\phantom{r} Instituto de Ciencias
Matem\'{a}ticas (CSIC - UAM - UC3M - UCM)
}
\email{albertoi@math.uc3m.es, flledo@math.uc3m.es, jmppardo@math.uc3m.es}


\date{\today; File: \textbf{\jobname.tex}}
\date{\today}

\thanks{The first and third name authors are partly supported by the project MTM2010-21186-C02-02
    of the spanish {\em Ministerio de Ciencia e Innovaci\'on} and QUITEMAD programme P2009 ESP-1594.
The second-named author was
partially supported by projects DGI MICIIN MTM2012-36372-C03-01 and Severo Ochoa SEV-2011-0087
of the spanish Ministry of Economy and Competition.
The third-named author was also partially supported in 2011 and 2012 by mobility grants of the
``\emph{Universidad Carlos III de Madrid}''
}

\begin{document}

\begin{abstract}
We construct in this article a class of closed semi-bounded quadratic forms
on the space of square integrable functions over a smooth Riemannian manifold
with smooth boundary. Each of these quadratic forms specifies a semi-bounded self-adjoint extension of the Laplace-Beltrami operator.
These quadratic forms are based on the  
Lagrange boundary form on the manifold
and a family of domains parametrized by a suitable class of unitary operators on the boundary that will be called admissible.
The corresponding quadratic forms are semi-bounded below and closable.
Finally, the representing operators correspond to semi-bounded self-adjoint extensions of the Laplace-Beltrami operator.
This family of extensions is compared with results existing in the literature and various examples and
applications are discussed.
\end{abstract}

\maketitle

\tableofcontents

\section{Introduction}

In this paper we construct a family of closed quadratic forms corresponding to a class of
self-adjoint extensions of the Laplace-Beltrami operator on a smooth Riemannian manifold with smooth boundary.
It is well-known that in a smooth manifold $\Omega$ with no boundary the minimal closed extension of the Laplace-Beltrami operator
$\Delta_\mathrm{min}$ is essentially self-adjoint.
However, if the manifold has a non-empty boundary $\partial \Omega$, then $\Delta_\mathrm{min}$
defines a closed and symmetric but {\em not self-adjoint} operator. Such situation is common in the study of quantum systems,
where some heuristic arguments suggest an expression for the Hamiltonian which is only symmetric.
The Laplace-Beltrami operator discussed here can be associated to free quantum systems on the manifold.
The description of such systems is not complete until a self-adjoint extension of the Laplace-Beltrami operator has been determined,
i.e.,  a Hamiltonian operator $H$. Only in this case a unitary evolution of the system is given, because of the one-to-one correspondence
between densely defined self-adjoint operators and strongly continuous one-parameter groups of unitary operators
$U_t = \exp itH$ provided by Stone's theorem. Therefore the specification of the self-adjoint extension is not just
a mathematical artifact, but an essential step in the description of the quantum mechanical system
(see, e.g., Chapter~X in \cite{reed-simon-2} for further results and motivation).

The collection of all self-adjoint extensions of a densely defined closed symmetric operator $T$ on a complex separable Hilbert
space $\mathcal{H}$
was described by von Neumann in terms of the isometries between the deficiency spaces $\mathcal{N}_\pm = \ker (T^\dagger \mp i I)$
of the operator $T$ (see, e.g., \cite{Ne29,reed-simon-2,We80}).
Unfortunately, beyond the one-dimensional case, the use of von Neumann's theorem to describe the self-adjoint
extensions of the Laplace-Beltrami operator is unfeasible. In fact, the computation of deficiency indices requires
the knowledge of the adjoint operator which is a difficult problem in itself (see \cite{behrnd-langer-10} and references therein).
Moreover, in von Neumann's classical result the use of important geometrical and physical data becomes rather indirect and for these
reasons the theory of extensions has been developed in many different ways and is still today an active research
area. The use of the Hermitian quadratic forms to address the extension problem has been one of the most
useful approaches since the pioneering work by Friedrichs (cf., \cite{Fr34}).
If $T$ is a symmetric and semi-bounded operator on the domain $\mathcal{D}(T)$, then the semi-bounded quadratic form
\begin{equation}\label{representation}
Q(\Phi) = \langle \Phi, T \Phi \rangle,  \qquad \Phi \in \mathcal{D}(T) \subset \H ,
\end{equation}
is closable and its closure is represented by a self-adjoint extension of $T$ with the same lower bound
(see, e.g., \cite{Ko99,reed-simon-2,We80}). Moreover, the domain of the closure of the quadratic
form satisfies a natural minimality condition. Kato's representation theorem provides the
characterization of closed semi-bounded quadratic forms as those that can be represented
by self-ajoint and semi-bounded operators as in \eqref{representation} (cf., \cite{kato:95,reed-simon-1,reed-simon-2}).
In the particular instance of the Laplace-Beltrami operator some of these closed
extensions on $\H=L^2(\Omega)$ are well known. The simplest examples are the quadratic forms associated to the Dirichlet and Neumann self-adjoint
extensions of the Laplacian: consider the positive and closed quadratic form
\begin{equation} \label{dirichlet}
Q(\Phi)=\norm{d\Phi}^2
\end{equation}
with domain $\D_D=\H^1_0(\Omega)$ in the Dirichlet case and domain
$\D_N=\H^1(\Omega) $ for the Neumann extension (see for instance \cite{davies:95} and references therein).
Also equivariant and Robin-type Laplacians can be naturally described in terms of closed and semi-bounded quadratic forms
(see, e.g., \cite{Grubb73,Grubb11,Kovarik-Laptev-12,lledo-post:07}). In this context the subtle relation between quadratic forms and representing
operators manifests through the fact that the form domain $\mathcal{D}(Q)$ always contains the operator domain $\mathcal{D}(T)$
of the representing operator. Therefore it is often possible to compare different form domains while the domains of the representing operators
remain unrelated. This fact allows, e.g., to develop spectral bracketing techniques in very different mathematical and physical situations
using the language of quadratic forms \cite{lledo-post:08,lledo-post:08b}.

In spite of the vast literature devoted to the subject, the determination of the
self-adjoint or, more generally, sectorial
extension of the operator and their spectral properties is still an active field of research
(see \cite{Arlinskii00,kato:95} and references therein).
Another example where the correct extension of a symmetric operator has been recently analyzed is
the case of the so called Berry's paradox when dealing with a class of Robin boundary conditions with a singular Dirichlet point \cite{Be08,Be09,Ma09}.
Hence the study of such quadratic forms is instrumental not only for the construction of a complete quantum system but for the analysis
of the spectrum of the corresponding self-adjoint Hamiltonian operators \cite{IbPer2010}.

The role of boundaries has been highlighted in the case of the study of self-adjoint extensions of formally self-adjoint differential
operators leading to the complete classification of boundary conditions by Grubb \cite{Grubb68} and to the theory of boundary triples
(see, e.g., \cite{BGP08} and references therein and \cite{behrnd-langer-07}
for the generalization to quasi-boudary triples; see also
Chapter~2 in \cite{post-lnm-12} for the description of boundary triples for quantum graphs and
\cite{post-13} for the theory of boundary pairs in the context of quadratic forms).

In a similar but slightly different direction focused on the physics of boundary dynamics
it was argued in \cite{AIM05} that self-adjoint extensions of the Bochner Laplacian are in one-to-one correspondence
with unitary operators on a Hilbert space of boundary data, the trace of the function and its normal derivative at the boundary.
Such characterization was shown to be particularly useful as it provides an explicit and easily workable description of the domain of the
corresponding self-adjoint extension by means of the condition, called in what follows boundary equation:
\begin{equation}\label{asorey1}
\varphi - i \dot\varphi = U (\varphi + i \dot\varphi) ,
\end{equation}
where $\varphi, \dot\varphi$ denote the trace at the boundary of a function
$\Phi$, i.e., $\varphi = \Phi\mid_{\partial\Omega}$ and its outward normal derivative
$\dot{\varphi} = \d\Phi (\nu )$, and $U$ is a unitary operator on the Hilbert space at the boundary
$L^2(\partial\Omega)$.
The analysis of such self-adjoint extensions by means of the corresponding quadratic forms leads (after integration by parts
on smooth functions) to the study of the quadratic form:
\begin{equation}\label{singular}
 Q(\Phi) = \norm{d\Phi}^2 - \langle \varphi, \dot\varphi \rangle .
 \end{equation}
Such quadratic form can be considered as a singular perturbation of the standard Dirichlet quadratic form \eqref{dirichlet}.
Unfortunately Koshmanenko's theorems \cite{Ko99} on closable singular perturbations of quadratic forms cannot be
directly applied to domains described by the boundary equation \eqref{asorey1} in general. Thus, in order to
characterize the domains of the self-adjoint extensions of the Laplace-Beltrami operator
determined by \eqref{asorey1} a different approach is needed.

In this article we present a self-contained analysis of the quadratic form \eqref{singular} on domains
satisfying Eq.~\eqref{asorey1}.
Their closability is proved under appropriate conditions on the unitary operator $U$ defining the extension.  Actually it is
shown that if $U$ has gap, i.e.,  if the eigenvalue $-1$ is isolated in its spectrum, and its partial Cayley transform is bounded in
the Sobolev norm $1/2$, the singularly perturbed Dirichlet quadratic form \eqref{singular} with domain determined by
condition \eqref{asorey1} is closable and semi-bounded below.
These results are obtained after a careful analysis of the domain defined by the boundary equation \eqref{asorey1}, the structure of the
radial Laplace operator defined on a collar neighborhood of the boundary, and a judiciously use of Neumann's extension of the given
quadratic form on the bulk of the manifold. As particular case these results include that all the Robin boundary conditions of the
form $\dot{\varphi}=g\varphi$\,, $g\in\C(\pO)$\,, lead to lower semi-bounded extensions of the Laplace-Beltrami operator.

The paper is organized as follows. Section~\ref{sectionNotation} is devoted to establish
basic definitions and results on quadratic forms and some technicalities on the Laplace-Beltrami operator
and Sobolev spaces in smooth manifolds with boundary. In Section~\ref{sec:class} we introduce the class of quadratic forms
whose closability and semi-boundedness will be established.
We will also specify the domains of the self-adjoint extensions
in terms of a class of maximal isotropic subspaces (cf., Theorem~\ref{teo:parametriceW} and Proposition~\ref{prop: asorey}).
The class of admissible unitary operators $U$ leading
to closable and semi-bounded quadratic forms is introduced at the end of this section paving the way to Section~\ref{sec:closable and semibounded qf},
where the main theorems proving the closability and semi-boundedness of the quadratic forms defined are discussed.
Finally, in Section~\ref{sec:examples} various families of examples with admissible unitaries at the boundary are obtained by using
several choices of values of the boundary data. For instance combining Dirichlet, Neumann and diverse identifications of subdomains of the boundary.

\section{Preliminaries: quadratic forms and the Laplace-Beltrami operator}\label{sectionNotation}

In this section we fix our notation and recall first some standard results of the theory of unbounded operators
and quadratic forms that will be useful later on. Standard references are,
e.g.,~\cite[Section~4.4]{davies:95}, \cite[Chapter~VI]{kato:95} or \cite[Section~VIII.6]{reed-simon-1}.
Then we will also introduce standard material on Riemannian manifolds with boundary,
the Laplace-Beltrami operator and the associated
Sobolev spaces. Some basic references for this part are, e.g., \cite{Ma01,Ad03,davies:95,Li72,Po81}.

\subsection{Quadratic forms and operators}\label{subsectionNotation}

\begin{definition}
Let $\D$ be a dense subspace of the Hilbert space $\H$ and denote by $Q\colon\mathcal{D}\times\mathcal{D}\to \mathbb{C}$
a sesquilinear form (anti-linear in the first entry and linear in the second entry). The quadratic form associated to $Q$
with domain $\mathcal D$ is its evaluation on the diagonal, i.e.,  $Q(\Phi):=Q(\Phi,\Phi)$,
$\Phi\in\mathcal{D}$. We say that the sesquilinear form is \textbf{Hermitian} if
\[
  Q(\Phi,\Psi)=\overline{Q(\Psi,\Phi)}\;,\quad \Phi,\Psi\in\mathcal{D}\;.
\]
The quadratic form is  \textbf{semi-bounded} if there is an $a\geq 0$ such that
\[
  Q(\Phi)\geq -a \norm{\Phi}^2\;,\; \Phi \in \mathcal{D}\;.
\]
The smallest possible value $a$ satisfying the preceding inequality is called the \textbf{lower bound} for the quadratic form $Q$. ç
In particular, if  $Q(\Phi)\geq 0$ for all $\Phi\in\mathcal{D}$ we say $Q$ is \textbf{positive}.

\end{definition}

Note that if $Q$ is semi-bounded with lower bound $a$, then $Q_a(\Phi):=Q(\Phi)+a\norm{\Phi}^2$, $\Phi\in\D$,
is positive on the same domain. We need to recall also the notions of closable and closed quadratic forms as
well as the fundamental representation theorems that relate closed semi-bounded quadratic forms with self-adjoint semi-bounded operators.

\begin{definition}
Let $Q$ be a semi-bounded quadratic form with lower bound $a\geq 0$ and dense domain $\D\subset\H$.
The quadratic form $Q$ is \textbf{closed} if $\D$ is closed with respect to the norm
\[
 \normm{\Phi}_Q:=\sqrt{Q(\Phi)+(1+a)\|\Phi\|^2}\;,\quad\Phi\in\D\;.
\]
If Q is closed and $\D_0\subset\D$ is dense with respect to the norm $\normm{\cdot}_Q$, then $\D_0$ is called a
\textbf{form core} for $Q$.
Conversely, the closed quadratic form $Q$ with domain $\D$ is called an
\textbf{extension} of the quadratic form $Q$ with domain $\D_0$. A quadratic form is said to be
\textbf{closable} if it has a closed extension.
\end{definition}

\begin{remark}$\phantom{=}$\label{Remclosable}
\begin{enumerate}
\item The norm $\normm{\cdot}_Q$ is induced by the following inner product on
the domain:
\[
 \langle\Phi,\Psi\rangle_Q:= Q(\Phi,\Psi)+(1+a)\langle\Phi,\Psi\rangle\;,\quad \Phi,\Psi\in\D\;.
\]
 \item The quadratic form $Q$ is closable iff whenever a sequence $\{\Phi_n\}_n\subset\D$ satisfies
$\norm{\Phi_n}\to 0$ and $Q(\Phi_n-\Phi_m)\to 0$, as $n,m\to\infty$, then $Q(\Phi_n)\to 0$.
 \item In general it is always possible to close $\D \subset\mathcal{H}$ with respect to the norm $\normm{\cdot}_Q$.
The quadratic form is closable iff this closure is a subspace of $\H$.
 \end{enumerate}

\end{remark}

\begin{theorem}[Kato's representation theorem]\label{fundteo}
Let $Q$ be an Hermitian, closed, semi-bounded quadratic form defined on the dense domain
$\D\subset\H$. Then it exists a unique, self-adjoint, semi-bounded operator $T$
with domain $\D(T)$ and the same lower bound such that:
\begin{enumerate}
\item $\Psi\in\mathcal{D}(T)$ iff $\Psi\in \D$ and it exists $\chi \in \H$ such that
$$Q(\Phi,\Psi)=\langle\Phi,\chi\rangle\,,\quad\forall \Phi\in\D\;.$$
In this case we write $T\Psi=\chi$.
\item $Q(\Phi,\Psi)=\langle\Phi,T\Psi\rangle$ for any $\Phi\in\D$,\;$\Psi\in\D(T)$.
\item $\D(T)$ is a core for $Q$.\\
\end{enumerate}
\end{theorem}

One of the most common uses of the representation theorem is to obtain self-adjoint extensions of symmetric, semi-bounded operators.
Given a semi-bounded, closed and symmetric operator $T$ one can consider the associated quadratic form
$$Q_T(\Phi,\Psi)=\scalar{\Phi}{T\Psi}\quad \Phi,\Psi\in\D(T)\;.$$ These quadratic forms are always closable,
cf., \cite[Theorem X.23]{reed-simon-2}, and therefore their closure is associated to a unique self-adjoint operator.
Even if the symmetric operator has infinite possible self-adjoint extensions, the representation theorem allows to select a particular one.
This extension is called the Friedrichs extension. The approach that we shall take in this article is close to this method.\\

\subsection{Scales of Hilbert Spaces}

Later on we will need the theory of scales of Hilbert spaces, also known as theory of rigged Hilbert spaces.
In the following paragraph we state the main results, (see, e.g., \cite{Be68,Ko99} for proofs and more results).

Let $\H$ be a Hilbert space with scalar product $\scalar{\cdot}{\cdot}$ and induced norm $\norm{\cdot}$. Let $\H_+$ be a dense linear
subspace of $\H$ which is a complete Hilbert space with respect to another scalar product that will be denoted by $\scalar{\cdot}{\cdot}_+$.
The corresponding norm is $\norm{\cdot}_+$ and we assume that
\begin{equation}\label{inclusion inequality}
	\norm{\Phi}\leq\norm{\Phi}_+\;,\quad \Phi\in\H_+\;.
\end{equation}

Any vector $\Phi\in\H$ generates a continuous linear functional $L_\Phi\colon\H_+\to \mathbb{C}$  as follows. For $\Psi\in\H_+$ define
\begin{equation}
	L_{\Phi}(\Psi)=\scalar{\Phi}{\Psi}\;.
\end{equation}
Continuity follows by the Cauchy-Schwartz inequality and Eq.\ \eqref{inclusion inequality}.
\begin{equation}
L_{\Phi}(\Psi) \leq \norm{\Phi}\cdot\norm{\Psi}\leq\norm{\Phi}\cdot\norm{\Psi}_+\;,
                    \quad\forall \Phi\in\H\;,\forall\Psi\in\H_+\;.	
\end{equation}
Since $L_\Phi$ represents a continuous linear functional on $\H_+$ it can be represented, according to Riesz theorem,
using the scalar product in $\H_+$. Namely, it exists a vector $\xi\in\H_+$ such that
\begin{equation}
	\forall\Psi\in\H_+\,,\quad L_{\Phi}(\Psi)=\scalar{\Phi}{\Psi}=\scalar{\xi}{\Psi}_+\;,
\end{equation}
and the norm of the functional coincides with the norm in $\H_+$ of the element $\xi$, i.e.,
 $$\norm{L_\Phi}=\sup_{\Psi\in\H_+}\frac{|L_\Phi(\Psi)|}{\norm{\Psi}_+}=\norm{\xi}_+\;.$$
One can use the above equalities to define an operator
\begin{equation}
\begin{array}{c}
\hat{I}\colon \H\to\H_+\\
\hat{I}\Phi=\xi\;.
\end{array}
\end{equation}
This operator is clearly injective since $\H_+$ is a dense subset of $\H$ and therefore it can be used to define a new scalar product on $\H$
\begin{equation}
	\scalar{\cdot}{\cdot}_-:=\scalar{\hat{I}\cdot}{\hat{I}\cdot}_+\;.
\end{equation}
The completion of $\H$ with respect to this scalar product defines a new Hilbert space, $\H_-$, and the corresponding norm will be
denoted accordingly by $\norm{\cdot}_-$. It is clear that $\H_+\subset\H\subset \H_-$ with dense inclusions.
Since $\norm{\xi}_+=\norm{\hat{I}\Phi}_+=\norm{\Phi}_-$, the operator $\hat{I}$ can be extended by continuity to an isometric bijection.

\begin{definition}
The Hilbert spaces $\H_+$, $\H$ and $\H_-$ introduced above define a \textbf{scale of Hilbert spaces}. The extension by continuity of the
operator $\hat{I}$ is called the \textbf{canonical isometric bijection}. It is denoted by:
\begin{equation}
I\colon \H_-\to\H_+\;.
\end{equation}
\end{definition}

\begin{proposition}\label{proppairing}
The scalar product in $\H$ can be extended continuously to a pairing
\begin{equation}
\pair{\cdot}{\cdot}\colon\H_-\times\H_+\to\mathbb{C}\;.
\end{equation}
\end{proposition}

\begin{proof}
Let $\Phi\in\H$ and $\Psi\in\H_+$. Using the Cauchy-Schwartz inequality we have the following
\begin{equation}\label{CSpairing}
	|\scalar{\Phi}{\Psi}|=|\scalar{I\Phi}{\Psi}_+|\leq \norm{I\Phi}_+\norm{\Psi}_+=\norm{\Phi}_-\norm{\Psi}_+\;.
\end{equation}
\end{proof}

\subsection{Laplace-Beltrami operator on Riemannian manifolds and Sobolev spaces}

Our aim is to describe a class of closable quadratic forms related to the self-adjoint extensions of the Laplace-Beltrami operator defined
on a compact Riemannian manifold. We shall start with the definition of such manifold and of the different spaces of functions that will
appear throughout the rest of this article.

Let $(\Omega,\pO,\eta)$ be a smooth, orientable, compact, Riemannian manifold  with metric
$\eta$ and smooth boundary $\partial \Omega$.
We will denote as  $\C^\infty (\Omega)$ the space of smooth functions of the Riemannian manifold $\Omega$ and
by $\C_c^\infty (\Omega)$ the space of smooth functions with compact support in the interior of $\Omega$.
The Riemannian volume form is written as $\d\mu_\eta$.

\begin{definition}
The \textbf{Laplace-Beltrami Operator} associated to the Riemannain manifold $(\Omega,\pO,\eta)$ is the
second order differential operator $\Delta_\eta:\C^\infty(\Omega)\to\C^\infty(\Omega)$ given by
$$\Delta_\eta\Phi=\frac{1}{\sqrt{|\eta|}}\frac{\partial}{\partial x�}\sqrt{|\eta|}\eta^{ij}\frac{\partial\Phi}{\partial x^j}\;.$$
\end{definition}

Let $(\tilde{\Omega},\tilde{\eta})$ be a smooth, orientable, boundaryless, compact Riemannian manifold with metric $\tilde{\eta}$.
The Laplace-Beltrami operator $-\Delta_{\tilde{\eta}}$ associated to the Riemannian manifold $(\tilde{\Omega},\tilde{\eta})$ defines a positive,
essentially self-adjoint, second order differential operator, cf. \cite{Ma01}. One can use it to define the following norms.

\begin{definition} \label{DefSobolev}
The \textbf{Sobolev norm of order $k$} in the boundaryless Riemannian manifold $(\tilde{\Omega},\tilde{\eta})$ is defined by
$$|| \Phi ||_k^2 := \int_{\tilde{\Omega}} \overline{\Phi} (I - \Delta_{\tilde{\eta}} )^{k}\Phi \d\mu_{\tilde{\eta}}\;. $$
The closure of the smooth functions with respect to this norm $\H^k(\tilde{\Omega}) := \overline{\C^{\infty}(\tilde{\Omega})}^{\norm{\cdot}_k}$
is the \textbf{Sobolev space of class $k$} of the Riemannian manifold $(\tilde{\Omega},\tilde{\eta})$\,.
The scalar products associated to these norms are written as $\scalar{\cdot}{\cdot}_k$.
In the case $k=0$ we will denote the $\H^0(\tilde{\Omega})$ scalar product simply by
$\scalar{\Phi}{\Psi}=\int_{\tilde{\Omega}} \overline{\Phi}\Psi \d\mu_{\tilde{\eta}}$.
\end{definition}

Note that Definition~\ref{DefSobolev} holds only for Riemannian manifolds without boundary.
The construction of the Sobolev spaces of functions over a manifold $(\Omega,\pO,\eta)$
cannot be done directly like in the definition above because the Laplace-Beltrami operator does not define in general
positive differential operators. However, it is possible to construct it as a quotient of the Sobolev space of functions
over a Riemannian manifold $(\tilde{\Omega},\tilde{\eta})$ without boundary.

\begin{definition}\label{DefSobolev2}
Let $(\Omega,\pO,\eta)$ be a Riemannian manifold and let $(\tilde{\Omega},\tilde{\eta})$ be any Riemannian manifold without boundary such that
$\smash{\overset{\scriptscriptstyle\circ}{\Omega}}$, i.e., the interior of $\Omega$, is an open submanifold of
$\tilde{\Omega}$. The \textbf{Sobolev space of class k} of the Riemannain manifold $(\Omega,\pO,\eta)$ is the quotient
$$\H^k(\Omega):=\H^k(\tilde{\Omega})/\{\Phi\in\tilde{\Omega}\mid \Phi|_\Omega=0\}\;. $$
The norm is denoted again as $\norm{\cdot}_k$. When there is ambiguity about the manifold, the subindex shall denote the full space, i.e.,
$$\norm{\cdot}_k=\norm{\cdot}_{\H^k(\Omega)}\;.$$
\end{definition}

It can be shown that the Sobolev spaces $\H^k(\Omega)$ do not depend on the particular choice of $\;\tilde{\Omega}$\,.
There are many equivalent ways to define the Sobolev norms. In particular we shall need the following characterization.

\begin{proposition}\label{equivalentsobolev}
The Sobolev norm of order 1, $\norm{\cdot}_1$, is equivalent to the norm $$\sqrt{\norm{\d\cdot}^2_{\Lambda^1}+\norm{\cdot}^2}\;,$$
where $\d$ stands for the exterior differential acting on functions, cf. \cite{Ma01}, and $\norm{\d\cdot}_{\Lambda^1}$ is
the induced norm from the natural scalar product among 1-forms $\alpha\in\Lambda^1(\Omega)$.
\end{proposition}

\begin{proof}
It is enough to show it for a boundaryless Riemannian manifold $(\tilde{\Omega},\tilde{\eta})\,.$
The Laplace-Beltrami operator can be expressed in terms of the exterior differential and its formal adjoint,
$$-\Delta_{\tilde{\eta}}=\d^\dagger\d\;,$$
where the formal adjoint is defined to be the unique differential operator $\d^\dagger:\Lambda^1(\tilde{\Omega})\to\C^\infty(\tilde{\Omega})$
that verifies $$\scalar{\alpha}{\d\Phi}_{\Lambda^1}=\scalar{\d^\dagger\alpha}{\Phi}\quad
\alpha\in\Lambda^1(\tilde{\Omega}),\Phi\in\C^\infty(\tilde{\Omega})\;.$$ Let $\Phi\in\C^{\infty}(\tilde{\Omega})$.
Then we have that
\begin{align*}
\norm{\Phi}^2_1&=\int_{\tilde{\Omega}}\bar{\Phi}(I-\Delta_{\tilde{\eta}})\Phi\d\mu_{\tilde{\eta}}\\
&=\int_{\tilde{\Omega}}\bar{\Phi}\Phi\d\mu_{\tilde{\eta}}+ \int_{\tilde{\Omega}}\bar{\Phi}\d^\dagger\d\Phi\d\mu_{\tilde{\eta}}\\
&=\norm{\Phi}^2+\scalar{\d\Phi}{\d\Phi}_{\Lambda^1}=\norm{\Phi}^2+\norm{\d\Phi}_{\Lambda^1}^2\;.
\end{align*}
\end{proof}
The subindex $\Lambda^1$ will be omitted when it is clear from the context which are the scalar products considered.\\

The boundary $\pO$ of the Riemannian manifold $(\Omega,\pO,\eta)$ has itself the structure of a Riemannian manifold
without boundary $(\pO,\partial\eta)$. The Riemannian metric induced at the boundary is just the pull-back of the
Riemannian metric $\partial\eta=i^{\star}\eta$, where $i:\pO\to\Omega$ is the inclusion map. The spaces of smooth functions
over the two manifolds verify that $\C^\infty(\Omega)\bigr|_{\pO}\simeq\C^{\infty}(\pO)$.

There is an important relation between the Sobolev spaces defined over the manifolds $\Omega$ and $\pO$.
This is the well known Lions trace theorem (cf. \cite[Theorem 7.39]{Ad03}, \cite[Theorem 8.3]{Li72}):

\begin{theorem}[Lions trace theorem]\label{LMtracetheorem}
Let $\Phi\in\C^{\infty}(\Omega)$ and let $\gamma:\C^\infty(\Omega)\to\C^{\infty}(\pO)$ be the trace map $\gamma(\Phi)=\Phi\bigr|_{\pO}$. There is a unique continuous extension of the trace map such that
\begin{enumerate}
\item $\gamma:\H^{k}(\Omega)\to\H^{k-1/2}(\pO)$, $k\geq 1/2$\;.\\
\item The map is surjective\;.\\
\end{enumerate}
\end{theorem}

Finally we introduce for later use some particular operators associated to the Laplacian.
Consider the symmetric operator on smooth functions with support away from the boundary
$\Delta_0:=\Delta_\eta\bigr|_{\C^\infty_c(\Omega)}$. Then we have the following extensions of it.

\begin{definition}\hfill
\begin{enumerate}
\item The \textbf{minimal closed extension} $\Delta_{\mathrm{min}}$ is defined to be the closure of $\Delta_0$.
Its domain is $\D(\Delta_{\mathrm{min}})=\H^2_0:=\overline{\C^\infty_c(\Omega)}^{\norm{\cdot}_2}$\,.
\item The \textbf{maximal closed extension} $\Delta_{\mathrm{max}}$ is the closed operator defined in the domain
$\D(\Delta_{\mathrm{max}})=\bigl\{ \Phi\in\H^0(\Omega)\bigr| \Delta_\eta\Phi\in\H^0(\Omega) \bigr\}$\,.
\end{enumerate}
\end{definition}

The trace map defined in Theorem \ref{LMtracetheorem} can be extended continuously to $\D(\Delta_{\mathrm{max}})$,
see for instance \cite{Fr05,Grubb68,Li72}:

\begin{theorem}[Weak trace theorem for the Laplacian]\label{weaktracetheorem}
The Sobolev space $\H^k(\Omega)$, with $k\geq2$, is dense in $\mathcal{D}(\Delta_{\mathrm{max}})$ and
there is a unique continuous extension of the trace map $\gamma$ such that
$$\gamma \colon \mathcal{D}(\Delta_{\mathrm{max}}) \to H^{-1/2}(\pO)\;.$$
Moreover $\ker \gamma = H_0^2(\Omega)$\,.
\end{theorem}

\section{A class of closable quadratic forms on a Riemannian manifold}\label{sec:class}

We begin presenting a canonical sesquilinear form that, on smooth functions over $\Omega$, is associated to
the Laplace-Beltrami operator. Motivated by this quadratic form we will address questions like hermiticity, closability
and semi-boundedness on suitable domains.

Integrating once by parts the expression
$\scalar{\Phi}{-\Delta_\eta\Psi}$ we obtain, on smooth functions, the following sesquilinear form
$Q \colon \C^{\infty}(\Omega)\times\C^{\infty}(\Omega)\to\mathbb{C}$\,,
\begin{equation} \label{Q-def}
Q(\Phi,\Psi)=\scalar{\d\Phi}{\d\Psi}_{\Lambda^1}-\scalarb{\varphi}{\dot{\psi}}\;.
\end{equation}

From now on the restrictions to the boundary are going to be denoted with the corresponding small size greek letters, $\varphi:=\gamma(\Phi)$.
The doted small size greek letters denote the restriction to the boundary of the normal derivatives, $\dot{\varphi}:=\gamma(\d\Phi(\nu))$,
where $\nu\in\mathfrak{X}(\Omega)$ is any vector field such that $\mathrm{i}_\nu\d\mu_\eta=\d\mu_{\partial\eta}$.
Notice that in the expression above $\d \Phi\in \Lambda^1(\Omega)$ is a 1-form on $\Omega$, thus the inner product $\langle \cdot, \cdot \rangle_{\Lambda^1}$
is defined accordingly by using the induced Hermitian structure on the cotangent bundle (see, e.g., \cite{Po81}).
We have therefore that $$\scalar{\d\Phi}{\d\Psi}_{\Lambda^1}= \int_\Omega \eta^{-1}(\d \bar{\Phi}, \d\Psi) \diff \mu_\eta\;.$$
In the second term at the right hand side of \eqref{Q-def} $\scalarb{\cdot}{\cdot}$ stands for the induced scalar product at the boundary
given explicitly by
\begin{equation}\label{scalarboundary}
\scalarb{\varphi}{\psi}=\int_{\pO }\bar{\varphi}\,  \psi \,\diff \mu_{\partial \eta}  ,
\end{equation}
where $\diff\mu_{\partial \eta}$ is the Riemannian volume defined by the restricted Riemannian metric $\partial\eta$. The subscript $\Lambda^1$ will be dropped from now on as along as there is no risk of confusion.\\\

In general, the sesquilinear form $Q$ defined above is not Hermitian. To study subspaces where $Q$ is Hermitian it is convenient to isolate the part of $Q$ related to the boundary data $(\varphi,\dot{\varphi})$.

\begin{definition}
Let $\Phi,\Psi\in \C^{\infty}(\Omega)$ and denote by $(\varphi,\dot{\varphi})$, $(\psi,\dot{\psi})$ the corresponding boundary data. The \textbf{Lagrange boundary form} is defined as:
\begin{equation}\label{lagrange}
	\Sigma\bigl(\Phi,\Psi\bigr)=\Sigma\bigl((\varphi,\dot{\varphi}),(\psi,\dot{\psi})\bigr):=\scalarb{\varphi}{\dot{\psi}}-\scalarb{\dot{\varphi}}{\psi}.
\end{equation}
Any dense subspace $\mathcal{D}\subset\H^0(\Omega)$ is said to be \textbf{isotropic with respect to $\Sigma$} if $\Sigma\bigl(\Phi,\Psi\bigr)=0\quad\forall\Phi,\Psi\in \D$.
\end{definition}

\begin{proposition}
The sesquilinear form $Q$ defined in Eq.~(\ref{Q-def}) on a dense subspace $\mathcal{D}\subset \H^{0}$ is Hermitian iff $\mathcal{D}$ is isotropic with respect to $\Sigma$.
\end{proposition}
\begin{proof}
The sesquilinear form $Q\colon\mathcal{D}\times\mathcal{D}\to\mathbb{C}$ is Hermitian if $Q(\Phi,\Psi)=\overline{Q(\Psi,\Phi)}$ for all $\Phi,\Psi\in\mathcal{D}$. By definition of $Q$ this is equivalent to $\Sigma\bigl(\Phi,\Psi\bigr)=0$, for all $\Phi,\Psi\in\mathcal{D}$, hence $\mathcal{D}$ is isotropic with respect to $\Sigma$. The reverse implication is obvious.
\end{proof}


\subsection{Isotropic subspaces}
The analysis of maximally isotropic subspaces can be handled more easily using the underlying Hilbert space struture of the Lagrange boundary form
and not considering for the moment any regularity question.
The expression \eqref{lagrange} can be understood as a sesquilinear form on the boundary Hilbert space $\H_b:=\H^0(\pO)\times \H^0(\pO)$\,,
\begin{equation*}
\Sigma\bigl(\Psi,\Phi\bigr)
	=\scalarb{\varphi}{\dot{\psi}}-\scalarb{\dot{\varphi}}{\psi}\;.\\
\end{equation*}
We will therefore focus now on the study of the sesquilinear form on the Hilbert space $\H_b$ and, while there is no risk of confusion, we will denote the scalar product in $\H^0(\pO)$ simply as $\scalar{\cdot}{\cdot}$\,,
\[
	\Sigma\left((\varphi_1,\varphi_2),(\psi_1,\psi_2)\right)
                :=\langle\varphi_1,\psi_2\rangle - \langle\varphi_2,\psi_1\rangle\;,\quad
	(\varphi_1,\varphi_2),(\psi_1,\psi_2)\in \H_b\,.
\]
Formally, $\Sigma$ is a sesquilinear symplectic form by which we mean that it satisfies the following conditions:
\begin{enumerate}
	\item $\Sigma$ is conjugate linear in the first argument and linear in the second.
	\item $\Sigma\Big((\varphi_1,\varphi_2),(\psi_1,\psi_2)\Big)
               =-\overline{\Sigma\Big((\psi_1,\psi_2),(\varphi_1,\varphi_2)\Big)}$\,,
                 $(\varphi_1,\varphi_2),(\psi_1,\psi_2)\in\H_b$\,.
	\item $\Sigma$ is nondegenerate, i.e.~$\Sigma\Big((\varphi_1,\varphi_2),(\psi_1,\psi_2)\Big)=0$
               for all $(\psi_1,\psi_2)\in\H_b$ implies $(\varphi_1,\varphi_2)=(0,0)$\,.
\end{enumerate}
The analysis of the isotropic subspaces of such sesquilinear forms is by no means new and their characterization is well known (\cite{BGP08}, \cite{Ko75}). However, in order to keep this article self-contained, we provide in the following paragraphs independent proofs of the main results that we will need.

First we write the sesquilinear symplectic form $\Sigma$ in diagonal form. This is done introducing the unitary Cayley transformation $\C\colon\H_b\to\H_b$\,,
\[
 \C(\varphi_1,\varphi_2):=\frac{1}{\sqrt{2}}\left( \varphi_1+\mathbf{i}\varphi_2, \varphi_1-\mathbf{i}\varphi_2 \right)\;,\;
  (\varphi_1,\varphi_2)\in\H_b\;.
\]
Putting
\[ \Sigma_c\Big((\varphi_+,\varphi_-),(\psi_+,\psi_-)\Big)
        :=-\mathbf{i}\Big( \langle\varphi_+,\psi_+\rangle - \langle\varphi_-,\psi_-\rangle\Big)
           \;,\quad (\varphi_+,\varphi_-),(\psi_+,\psi_-)\in\H_b \;,
\]
the relation between $\Sigma$ and $\Sigma_c$ is given by
\begin{equation}\label{relation}
	\Sigma\Big((\varphi_1,\varphi_2),(\psi_1,\psi_2)\Big)
  	=\Sigma_c\Big(\C(\varphi_1,\varphi_2),\C(\psi_1,\psi_2)\Big)
	\;,\quad (\varphi_1,\varphi_2),(\psi_1,\psi_2)\in\H_b\;.
\end{equation}

\begin{definition}
Consider a subspace $\mathcal{W}\subset\H_b$ and define the \textbf{$\Sigma$-orthogonal subspace} by
\[
\mathcal{W}^{\perp_\Sigma}:=\left\{(\varphi_1,\varphi_2)\in\H_b\mid\Sigma\Big((\varphi_1,\varphi_2),(\psi_1,\psi_2)\Big)=0
                  \;,\;\forall(\psi_1,\psi_2)\in \mathcal{W}\right\}\;.
\]
A subspace $\mathcal{W}\subset\H_b$ is \textbf{$\Sigma$-isotropic} [resp.~\textbf{maximally $\Sigma$-isotropic}] if $\mathcal{W}\subset \mathcal{W}^{\perp_\Sigma}$ [resp.~$\mathcal{W}= \mathcal{W}^{\perp_\Sigma}$].
\end{definition}

We begin enumerating some direct consequences of the preceding definitions:
\begin{lemma}
Let $\mathcal{W}\subset\H_b$ and put $\mathcal{W}_c:=\C(\mathcal{W})$.
	\begin{enumerate}\label{lemma-iso}
 		\item $\mathcal{W}$ is $\Sigma$-isotropic [resp.~maximally $\Sigma$-isotropic] iff $\mathcal{W}_c$ is $\Sigma_c$-isotropic [resp.~maximally $\Sigma_c$-isotropic].
		 \item If $(\varphi_1,\varphi_2)\in \mathcal{W}\subset \mathcal{W}^{\perp_\Sigma}$, then $\langle \varphi_1,\varphi_2\rangle=\overline{\langle \varphi_1,\varphi_2\rangle}$. If $(\varphi_+,\varphi_-)\in \mathcal{W}_c\subset \mathcal{W}_c^{\perp_{\Sigma_c}}$, then $\|\varphi_+\|=\|\varphi_-\|$.
	\end{enumerate}
\end{lemma}
\begin{proof}
Part (i) follows directly from Eq.(\ref{relation}) and the fact that $\C$ is a unitary transformation. To prove (ii) note that if $(\varphi_1,\varphi_2)$ is in an isotropic subspace $\mathcal{W}$, then
\[
 \Sigma\Big((\varphi_1,\varphi_2),(\varphi_1,\varphi_2)\Big)
         =\langle\varphi_1,\varphi_2\rangle - \langle\varphi_2,\varphi_1\rangle=0\;.
\]
One argues similarly in the other case.
\end{proof}

\begin{proposition}\label{pro:W}
Let $\mathcal{W}_\pm\subset\H^0(\pO)$ be closed subspaces and put $\mathcal{W}_c:=\mathcal{W}_+\times \mathcal{W}_-\subset\H_b$.
\begin{enumerate}
	\item The subspace $\mathcal{W}_c$ is $\Sigma_c$-isotropic iff it exists a partial isometry
            $V\colon \H^0(\pO)\to \H^0(\pO)$ with initial space $\mathcal{W}_+$
            and final space $\mathcal{W}_-$, i.e.~$V^*V(\H^0(\pO))=\mathcal{W}_+$ and $VV^*(\H^0(\pO))=\mathcal{W}_-$
            and
\[
 \mathcal{W}_c=\{(\varphi_+,V\varphi_+)\mid \varphi_+\in \mathcal{W}_+\}=\mathrm{gra}\,V\;.
\]
	\item The subspace $\mathcal{W}_c$ is maximally $\Sigma_c$-isotropic iff it exists a unitary $U\colon \H^0(\pO)\to \H^0(\pO)$ such that
		\begin{equation}\label{eq:U}
			\mathcal{W}_c=\{(\varphi_+,U\varphi_+)\mid \varphi_+\in \H^0(\pO)\}=\mathrm{gra}\,U\;.
		\end{equation}
\end{enumerate}

\end{proposition}
\begin{proof}
(i) For any $(\varphi_+,\varphi_-)\in \mathcal{W}_c$ we define the mapping
$V\colon \H^0(\pO)\to \H^0(\pO)$ by $V(\varphi_+):=\varphi_-$, $\varphi_+\in \mathcal{W}_+$ and $V(\varphi)=0$, $\varphi\in \mathcal{W}_+^\perp$. Since $\mathcal{W}_c\subset \mathcal{W}_c^{\perp_{\Sigma_c}}$ we have from part (ii) of Lemma~\ref{lemma-iso} that $V$ is a well-defined linear map and a partial isometry. The reverse implication is immediate: for any $(\varphi_+,V\varphi_+)\in \mathcal{W}_c$ we have
\[
 \Sigma_c\Big((\varphi_+,V\varphi_+),(\psi_+,V\psi_+)\Big)
     =-\mathbf{i}\Big(\langle\varphi_+,\psi_+\rangle- \langle V\varphi_+,V\psi_+\rangle\Big)
     =0\;,\quad \psi_+\in \H^0(\pO)\;,
\]
hence, $ \mathcal{W}_c=\{(\varphi_+,V\varphi_+)\mid \varphi_+\in \mathcal{W}_+\}=\mathrm{gra}\,V$, is $\Sigma_c$-isotropic.

(ii) Supose that $\mathcal{W}_c= \mathcal{W}_c^{\perp_{\Sigma_c}}$. By the previous item we have $\mathcal{W}_c=\{(\varphi_+,U\varphi_+)\mid \varphi_+\in \mathcal{W}_+\}$ for some partial isometry $U\colon \H^0(\pO)\to \H^0(\pO)$. Consider the following decompositions $\H^0(\pO)=\mathcal{W}_+\oplus \mathcal{W}_+^\perp= (U\mathcal{W}_+)\oplus (U\mathcal{W}_+)^\perp$ and note that any $(\varphi_+^\perp,\varphi_-^\perp)\in \mathcal{W}_+^\perp\times (U\mathcal{W}_+)^\perp$ satisfies $(\varphi_+^\perp,\varphi_-^\perp)\in \mathcal{W}_c^{\perp_{\Sigma_c}}$. Since $\mathcal{W}_c= \mathcal{W}_c^{\perp_{\Sigma_c}}$ we must have $\varphi_+^\perp=\varphi_-^\perp=0$, or, equivalently, $\mathcal{W}_+=\H^0(\pO)=U\mathcal{W}_+$, hence $\ker U =\ker U^*=\{0\}$ and $U$ is a unitary map.

To prove the reverse implication consider $\mathcal{W}_c=\{(\varphi_+,U\varphi_+)\mid \varphi_+\in \H^0(\pO)\}$ with $U$ unitary and choose $(\psi_+,\psi_-)\in \mathcal{W}_c^{\perp_{\Sigma_c}}$. Then for any $\varphi_+\in \H^0(\pO)$ we have
\[
 0=\Sigma_c\Big((\varphi_+,U\varphi_+),(\psi_+,\psi_-)\Big)
  =-\mathbf{i}\Big(\langle \varphi_+,\psi_+\rangle- \langle U\varphi_+,\psi_-\rangle\Big)
  =-\mathbf{i}\Big(\langle \varphi_+,(\psi_+-U^*\psi_-)\rangle\Big)\;.
\]
This shows that $\psi_-=U\psi_+$ and hence $(\psi_+,\psi_-)\in \mathcal{W}_c$, therefore $\mathcal{W}_c$ is maximally $\Sigma_c$-isotropic.
\end{proof}

The previous analysis allows to characterize finally the $\Sigma$-isotropic subspaces of the boundary Hilbert space $\H_b$.
\begin{theorem}\label{teo:parametriceW}
 A closed subspace $\mathcal{W}\subset\H_b$ is maximally $\Sigma$-isotropic iff there exists a unitary $U\colon  \H^0(\pO)\to \H^0(\pO)$ such that
\[
 \mathcal{W}=\left\{\Big((\1+U)\varphi\,,\,-\mathbf{i}(\1- U)\varphi\Big)\mid \varphi\in \H^0(\pO))\right\}\;.
\]
\end{theorem}
\begin{proof}
By Lemma~\ref{lemma-iso}~(i) and Propostion~\ref{pro:W}~(ii) we have that $\mathcal{W}$ is maximally $\Sigma$-isotropic iff $\mathcal{W}=\C^{-1}\mathcal{W}_c$, where $\mathcal{W}_c$ is given by Eq.~(\ref{eq:U}).
\end{proof}

\begin{proposition}\label{prop: asorey}
Let $U\colon  \H^0(\pO)\to \H^0(\pO)$ be a unitary operator and consider the maximally isotropic subspace $\mathcal{W}$ given in Theorem~\ref{teo:parametriceW}. Then $\mathcal{W}$ can be rewritten as
		\begin{equation}\label{eq:asorey}
			\mathcal{W}=\Big\{(\varphi_1\,,\,\varphi_2)\in\H_b\mid \varphi_1-\mathbf{i}\varphi_2= U(\varphi_1+\mathbf{i}\varphi_2)\Big\}\;.
		\end{equation}
\end{proposition}

\begin{proof}
Let $\mathcal{W}$ be given as in Theorem~\ref{teo:parametriceW} and let $\mathcal{W}'$ be a subspace defined as in Eq.~\eqref{eq:asorey}. Put $\varphi_1:=(\1+U)\varphi$ and $\varphi_2:=-\mathbf{i}(\1-U)\varphi$. Then it is straightforward to verify that $(\varphi_1\,,\,\varphi_2)$ satisfy the relation defining Eq.~ \eqref{eq:asorey} and therefore $\mathcal{W}\subset \mathcal{W}'$.

Consider a subspace $\mathcal{W}'$ defined as in Eq.~\eqref{eq:asorey} and let $(\varphi_1\,,\,\varphi_2)\in \mathcal{W}'$. Then the following relation holds
\begin{equation}\label{WW'1}
	(1-U)\varphi_1-\mathbf{i}(1+U)\varphi_2=0\;.
\end{equation}
Now consider that $(\varphi_1,\varphi_2)\in \mathcal{W}^\perp$. Then for all $\varphi\in\H^0(\pO)$
\begin{align*}
	0	&=\scalar{\varphi_1}{(1+U)\varphi}+\scalar{\varphi_2}{-\mathbf{i}(1-U)\varphi}\\
		&=\scalar{(1+U^*)\varphi_1+\mathbf{i}(1-U^*)\varphi_2}{\varphi}
\end{align*}
and therefore
\begin{equation}\label{WW'2}
	(1+U^*)\varphi_1+\mathbf{i}(1-U^*)\varphi_2=0\;.
\end{equation}
Now we can arrange Eqs.~\eqref{WW'1} and \eqref{WW'2}
\begin{equation}\label{M:equation}
M\begin{pmatrix}
\varphi_1\\\varphi_2
\end{pmatrix}
:=
\begin{pmatrix}
1-U & -\mathbf{i}(1+U)\\1+U^* & \mathbf{i}(1-U^*)
\end{pmatrix}
\begin{pmatrix}
\varphi_1\\\varphi_2
\end{pmatrix}
=0\;,
\end{equation}
where now $M\colon \H_b\to\H_b$. But clearly $M$ is a unitary operator so that Eq.~\eqref{M:equation} implies that $(\varphi_1,\varphi_2)=0$ and therefore $\mathcal{W}\oplus \mathcal{W}'^\perp=(\mathcal{W}^\perp\bigcap \mathcal{W}')^\perp=\H_b$. This condition together with $\mathcal{W}\subset \mathcal{W}'$ implies $\mathcal{W}=\mathcal{W}'$ because $\mathcal{W}$ is a closed subspace, as it is easy to verify.
\end{proof}

\subsection{Admissible unitaries and closable quadratic forms}

In this subsection we will restrict to a family of unitaries $U\colon  \H^0(\pO)\to \H^0(\pO)$ that will allow us to describe a wide class
of quadratic forms whose Friedrichs' extensions are associated to self-adjoint extensions of the Laplace-Beltrami operator.
\begin{definition}\label{DefGap}
Let $U\colon  \H^0(\pO)\to\H^0(\pO)$ be unitary and denote by $\sigma(U)$ its spectrum. We say that the unitary $U$ on the boundary
\textbf{has gap at $-1$} if one of the following conditions hold:
\begin{enumerate}
\item $\1+U$ is invertible.
\item $-1\in\sigma(U)$ and $-1$ is not an accumulation point of $\sigma(U)$.
\end{enumerate}
\end{definition}

\begin{definition}\label{P,boundary}
Let $U$ be a unitary operator acting on $\H^0(\pO)$ with gap at $-1$.
Let $E_\lambda$ be the spectral resolution of the identity associated to the unitary $U$,
i.e.,  $$U=\int_{[0,2\pi]}e^{\mathbf{i}\lambda}\d E_\lambda\;.$$
The \textbf{invertibility boundary space} $W$ is defined by
$W=\operatorname{Ran}E^{\bot}_{\{\pi\}}\,.$ The orthogonal projection onto $W$ is denoted by $P$.
\end{definition}

\begin{definition}\label{partialCayley}
Let $U$ be a unitary operator acting on $\H^0(\pO)$ with gap at $-1$. The \textbf{partial Cayley transform} $A_U:\H^0(\pO)\to W$ is
the operator
\[
A_U:=\mathbf{i}\,P (U-\mathbb{I}) (U+\mathbb{I})^{-1}\;.
\]

\end{definition}

\begin{proposition}
The partial Cayley transform is a bounded, self-adjoint operator on $\H^0(\pO)$.
\end{proposition}

\begin{proof}
First notice that the operators $P$, $U$ and $A_U$ commute. That $A_U$ is bounded is a direct consequence of the Definition~\ref{DefGap},
because the operator $P(\mathbb{I}+U)$ is under these assumptions an invertible bounded operator on the boundary space $W$.
To show that $A_U$ is self-adjoint
consider the spectral resolution of the identity of the operator $U$. Since $U$ has gap at $-1$,
either $\{e^{\mathbf{i\pi}}\}\not\in\sigma(U)$ or there
exists a neighborhood $V$ of $\{e^{\mathbf{i}\pi}\}$ such that it does not contain any element of the spectrum $\sigma(U)$ besides
$\{e^{\mathbf{i}\pi}\}$. Pick $\delta\in V\cap S^1$. Then one can express the operator $A_U$ using the spectral resolution of the identity of the
operator $U$ as
$$A_U=\int_{-\pi+\delta}^{\pi-\delta}\mathbf{i}\frac{e^{\mathbf{i}\lambda}-1}{e^{\mathbf{i}\lambda}+1}\d E_\lambda
=\int_{-\pi+\delta}^{\pi-\delta}-\tan\frac{\lambda}{2}\d E_\lambda\;.$$
Since $\lambda\in[-\pi+\delta,\pi-\delta]$, then $\tan\frac{\lambda}{2}\in\mathbb{R}$.
Therefore the spectrum of $A_U$ is a subset of the real line, necessary and sufficient condition for a closed, symmetric operator to be self-adjoint.
\end{proof}

We can now introduce the class of closable quadratic forms that was announced at the beginning of this section.

\begin{definition}\label{DefQU}
Let $U$ be a unitary with gap at $-1$, $A_U$ the corresponding partial Cayley transform and $\gamma$
the trace map considered in Theorem~\ref{LMtracetheorem}.
The Hermitian quadratic form associated to the unitary $U$ is defined by
$$Q_U(\Phi,\Psi)=\scalar{\d\Phi}{\d\Psi}-\scalarb{\gamma(\Phi)}{A_U\gamma(\Phi)}\;.$$
on the domain
$$\D_U=\bigl\{ \Phi\in\H^1(\Omega)\bigr|P^{\bot}\gamma(\Phi)=0 \bigr\}\;.$$
\end{definition}

\begin{proposition}\label{H1bound}
The quadratic form $Q_U$ is bounded by the Sobolev norm of order 1, $$Q_U(\Phi,\Psi)\leq K\norm{\Phi}_1\norm{\Psi}_1\;.$$
\end{proposition}

\begin{proof}
That the first summand of $Q_U$
is bounded by the $\H^1(\Omega)$ norm is direct consequence of the Cauchy-Schwartz inequality and
Proposition~\ref{equivalentsobolev}.

For the second term we have that
\begin{align*}
|\scalarb{\gamma(\Phi)}{A_U\gamma(\Psi)}|	&\leq \norm{A_U}\cdot \norm{\gamma(\Phi)}_0\,\norm{\gamma(\Psi)}_0\\
&\leq C \norm{A_U}\cdot\norm{\gamma(\Phi)}_{\frac{1}{2}}\,\norm{\gamma(\Psi)}_{\frac{1}{2}}\\
&\leq C' \norm{A_U}\cdot\norm{\Phi}_{1}\norm{\Psi}_{1}\;,
\end{align*}
where we have used Theorem~\ref{LMtracetheorem} in the last inequality.
\end{proof}

Finally, we need an additional condition of admissibility on the unitaries on the boundary that will be needed
to prove the closability of $Q_U$.

\begin{definition}\label{def:admissible}
Let $U$ be a unitary with gap at $-1$. The unitary is said to be \textbf{admissible} if the partial Cayley transform
$A_U\colon\H^0(\pO)\to \H^0(\pO)$ is continuous with respect to the Sobolev norm of order $1/2$, i.e.,
$$\norm{A\varphi}_{\H^{1/2}(\pO)}\leq K \norm{\varphi}_{\H^{1/2}(\pO)}\;.$$
\end{definition}

\begin{example}
Consider a manifold with boundary given by the unit circle, i.e.,  $\partial\Omega=S^1$,
and define the unitary $(U_\beta\varphi)(z):=e^{i\beta(z)}\,\varphi(z)$, $\varphi\in L^2(S^1)$.
If $\beta\in L^2(S^1)$ and $\ran\beta\subset\{\pi\}\cup [0,\pi-\delta]\cup [\pi+\delta,2\pi)$, for some $\delta >0$, then $U_\beta$ has gap
at $-1$. If, in addition, $\beta\in C^\infty(S^1)$, then $U_\beta$ is admissible.
\end{example}


\section{Closable and semi-bounded quadratic forms}\label{sec:closable and semibounded qf}

This section addresses the questions of semi-boundedness and closability of
the quadratic form $Q_U$ defined on its domain $\D_U$
(cf.~Definition~\ref{DefQU}).

\subsection{Functions and operators on collar neighborhoods}

We will need first some technical results that refer to the functions and operators in a collar neighborhood
close to the boundary $\pO$ and that we will denote by $\Xi$.
Recall the conventions at the beginning of Section~\ref{sec:class}: if $\Phi\in\H^1(\Omega)$,
then $\varphi=\gamma(\Phi)$ denotes its restriction to $\pO$ and for $\Phi$ smooth, $\dot{\varphi}$ is the restriction to
the boundary of the normal derivative.

\begin{lemma}\label{Lemma approxdotphi}
Let $\Phi\in\H^1(\Omega)$, $f\in\H^{1/2}(\pO)$. Then, for every $\epsilon>0$ it exists $\tilde{\Phi}\in\C^\infty(\Omega)$ such that
$\norm{\Phi-\tilde{\Phi}}_1<\epsilon$, $\norm{\varphi-\tilde{\varphi}}_{\H^{1/2}(\pO)}<\epsilon$ and
$\norm{f-\dot{\tilde{\varphi}}}_{\H^{1/2}(\pO)}<\epsilon$\,.
\end{lemma}

\begin{proof}
The first two inequalities are standard (cf., Theorem~\ref{LMtracetheorem}). Moreover, it
is enough to consider $\Phi\in\C^{\infty}(\Omega)$ with
$\d\Phi(\nu)\equiv0$, where $\nu\in\mathfrak{X}(\Omega)$ is the normal vector field, on a collar neighborhood $\Xi$ of $\pO$,
(see \cite[Chapter 4]{Hi76} for details on such neighbourhoods).
According to the proof of \cite[Theorem 7.2.1]{davies:95} this is a dense subset of $\H^1(\Omega)$.
The compactness assumption of $\Omega$ assures that the collar neighborhood has a minimal width $\delta$. Without loss of
generality we can consider that the collar neighborhood $\Xi$ has gaussian coordinates $\mathbf{x}=(r,\bt)$, being $\frac{\partial}{\partial r}$
the normal vector field pointing outwards. In particular,
we have that  $\Xi\simeq[-\delta,0]\times\pO$ and $\pO\simeq\{0\}\times\pO$. It is enough to consider $f\in\H^{1}(\pO)$, because $\H^1(\pO)$ is dense in $\H^{1/2}(\pO)$.

Consider a smooth function $g\in\C^\infty(\mathbb{R})$ with the following properties:
\begin{itemize}
\item $g(0)=1$ and $g'(0)=-1$\,.
\item $g(s)\equiv0$, $s\in[2,\infty)$\,.
\item $|g(s)|<1$ and $|g'(s)|<1$\,.
\end{itemize}

Define now the rescaled functions
$g_n(r):=\frac{1}{n}g(-nr)$. Let $\{f_n(\bt)\}_n\subset\C^\infty(\pO)$ be any sequence such that $\norm{f_n-f}_{\H^{1}(\pO)}\to 0$.
Now consider the smooth functions
\begin{equation}\label{smoothphi}
\tilde{\Phi}_n(\bx):=\Phi(\mathbf{x})+g_n(r)f_n(\bt)\;.
\end{equation}
Clearly we have that $\dot{\tilde{\varphi}}_n(\bt)\equiv f_n(\bt)$ and therefore $\norm{\dot{\tilde{\varphi}}_n-f}_{\H^1(\pO)}\to 0$ as needed.
Now we are going to show that $\tilde{\Phi}_n\stackrel{\H^1}{\to}\Phi$. According to Proposition \ref{equivalentsobolev} it is enough to show that
the functions and all their first derivatives converge in the $\H^0(\Omega)$ norm.
\begin{subequations}\label{H1convergence}
\begin{equation}
\norm{\tilde{\Phi}_n(\mathbf{x})-\Phi(\bx)}_{\H^0(\Omega)}=\norm{g_n(r)f_n(\bt)}_{\H^0([-\frac{2}{n},0]\times\pO)}
\leq\frac{2}{n^2}\norm{f_n}_{\H^0(\pO)}\;.
\end{equation}
\begin{equation}
\norm{\frac{\partial}{\partial r}\tilde{\Phi}_n(\mathbf{x})-\frac{\partial}{\partial r}\Phi(\bx)}_{\H^0(\Omega)}
\leq\norm{f_n(\bt)}_{\H^0([-\frac{2}{n},0]\times\pO)}\leq\frac{2}{n}\norm{f_n}_{\H^0(\pO)}\;.
\end{equation}
\begin{equation}
\norm{\frac{\partial}{\partial \theta}\tilde{\Phi}_n(\mathbf{x})-\frac{\partial}{\partial \theta}\Phi(\bx)}_{\H^0(\Omega)}
=\norm{g_n(r)\frac{\partial}{\partial \theta}f_n(\bt)}_{\H^0([-\frac{2}{n},0]\times\pO)}\leq\frac{2C}{n^2}\norm{f_n}_{\H^1(\pO)}\;.
\end{equation}
\end{subequations}

The constant $C$ in the last inequality comes from $\norm{\partial_\theta f_n}_{\H^0(\pO)}\leq C \norm{f_n}_{\H^1(\pO)}$. Since $\{f_n(\bt)\}$ is a convergent sequence in $\H^1(\pO)$
the norms appearing at the right hand sides are bounded.
\end{proof}

\begin{corollary}\label{Corsemibounded}
Let $\Phi\in\H^1(\Omega)$ and $c\in\mathbb{R}$. Then for every $\epsilon>0$
there exists a
$\tilde{\Phi}\in\C^{\infty}(\Omega)$ with $\dot{\tilde{\varphi}}=c\,\tilde{\varphi}$ such that
$\norm{\Phi-\tilde{\Phi}}_1<\epsilon$.
\end{corollary}

\begin{proof}
As in the proof of the preceding lemma it is enough to approximate any smooth function
$\Phi$ with vanishing normal derivative in a collar neighborhood.
Pick now a sequence of smooth functions $$\tilde{\Phi}_n(\bx):=\Phi(\bx)+c\Phi(0,\bt)\bigl(g_n(r)-\frac{1}{n}\bigr)\;,$$
where $g_n$ is the sequence of scaled functions defined in the proof of the preceding lemma.
This family of functions clearly verifies the boundary condition
$\dot{\tilde{\varphi}}=c\,\tilde{\varphi}$.
The inequalities \eqref{H1convergence} now read
\begin{align*}
\norm{\tilde{\Phi}_n(\mathbf{x})-\Phi(\bx)}_{\H^0(\Omega)}
&\leq\norm{c\Phi(0,\bt)\bigl(g_n(r)-\frac{1}{n}\bigr)}_{\H^0([-\frac{2}{n},0]\times\pO)}
  +\frac{c}{n}\vol(\Omega)\cdot\sup_{\Omega}|\Phi(0,\bt)|\\
&\leq\frac{2}{n^2}\norm{c\Phi(0,\bt)}_{\H^0(\pO)}+\frac{c}{n}\vol(\Omega)\cdot\sup_{\Omega}|\Phi(0,\bt)|\;.
\end{align*}
\begin{align*}
\norm{\frac{\partial}{\partial r}\tilde{\Phi}_n(\mathbf{x})-\frac{\partial}{\partial r}\Phi(\bx)}_{\H^0(\Omega)}
&\leq\norm{c\Phi(0,\bt)}_{\H^0([-\frac{2}{n},0]\times\pO)}\\
&\leq\frac{2c}{n}\norm{\Phi(0,\bt)}_{\H^0(\pO)}\;.
\end{align*}
\begin{align*}
\norm{\frac{\partial}{\partial \theta}\tilde{\Phi}_n(\mathbf{x})-\frac{\partial}{\partial \theta}\Phi(\bx)}_{\H^0(\Omega)}
&\leq\norm{c\bigl(g_n(r)-\frac{1}{n}\bigr)\frac{\partial}{\partial \theta}\Phi(0,\bt)}_{\H^0([-\frac{2}{n},0]\times\pO)}
  +\frac{c}{n}\vol(\Omega)\cdot\sup_{\Omega}| \frac{\partial \Phi(0,\bt)}{\partial \theta}|\\
&\leq\frac{2c}{n^2}\norm{\frac{\partial \Phi(0,\bt)}{\partial \theta}}_{\H^0(\pO)}
  +\frac{c}{n}\vol(\Omega)\cdot\sup_{\Omega}|\frac{\partial \Phi(0,\bt)}{\partial \theta}|\;.
\end{align*}
\end{proof}

\begin{corollary}\label{Corclosable}
Let $\{\Phi_n\}_n\subset\H^1(\Omega)$ and $A_U$ be the partial Cayley transform of an admissible unitary $U$.
Then it exists a sequence of smooth functions $\{\tilde{\Phi}_n\}\subset\C^{\infty}(\Omega)$ such that
$\norm{\Phi_n-\tilde{\Phi}_n}_{\H^1(\Omega)}<\frac{1}{n}$\,,
$\norm{\varphi_n-\tilde{\varphi}_n}_{\H^{1/2}(\pO)}<\frac{1}{n}$\,, and
$\norm{\dot{\tilde{\varphi}}_n-A_U\tilde{\varphi}_n}_{\H^{1/2}(\pO)}<\frac{1}{n}$\,.
\end{corollary}

\begin{proof}
For $\Phi_{n_0}$, $n_0\in\mathbb{N}$, take the approximating smooth function
$\tilde{\Phi}_{n_0}$ as in Lemma~\ref{Lemma approxdotphi} with
$$f:=A_U\varphi_{n_0}\in\H^{1/2}(\pO)$$
(note that since $U$ is admissible we have indeed that $f\in\H^{1/2}(\pO)$, cf. Definition~\ref{def:admissible}).
Choose also $\epsilon>0$ such that
$$\epsilon\leq \frac{1}{(1+\norm{A_U}_{\H^{1/2}(\pO)})n_0}$$
and note that this implies $\epsilon\leq\frac{1}{n_0}$.
Then the first two inequalities follow directly from Lemma~\ref{Lemma approxdotphi}.
Moreover, we also have
\begin{align*}
\norm{\dot{\tilde{\varphi}}_{n_0}-A_U\tilde{\varphi}_{n_0}}_{\H^{1/2}(\pO)}&\leq \norm{\dot{\tilde{\varphi}}_{n_0}-A_U\varphi_{n_0}}_{\H^{1/2}(\pO)}
  +\norm{A_U\varphi_{n_0}-A_U\tilde{\varphi}_{n_0}}_{\H^{1/2}(\pO)}\\
&\leq \epsilon+ \norm{A_U}_{\H^{1/2}(\pO)}\norm{\varphi_{n_0}-\tilde{\varphi}_{n_0}}_{\H^{1/2}(\pO)}\leq (1+\norm{A_U}_{\H^{1/2}(\pO)})\epsilon
 \leq \frac{1}{n_0}\;.
\end{align*}
\end{proof}

For the analysis of the semi-boundedness and closability of the quadratic form $(Q_U,\D_U)$ defined in the previous section
we need to analyze first the following one-dimensional problem in an interval. The operator is defined with Neumann conditions on one end of the interval and Robin-type conditions on the other end.

\begin{definition}\label{Defunidimensional}
Consider the interval $I=[0,2\pi]$ and a real constant $c\in\mathbb{R}$.
Define the second order differential operator
$$  R\colon\D(R)\to\H^0([0,2\pi]) \quad\text{by}\quad
    R=-\frac{\d^2}{\d r^2}
$$
on the domain
$$\D(R):=\left\{\Phi\in\C^{\infty}(I)\;\Bigr|\;\; \frac{\partial \Phi}{\partial r}\bigr|_{r=0}=0
         \quad\text{and}\quad \frac{\partial \Phi}{\partial r}\bigr|_{r=2\pi}
=c\Phi|_{r=2\pi} \right\}\subset\H^0([0,2\pi])\;.$$
\end{definition}

\begin{proposition}\label{prop: intervalrobin}
The symmetric operator $R$ of Definition~\ref{Defunidimensional} is essentially self-adjoint with discrete
spectrum and semi-bounded below with lower bound $\Lambda_0$\,.
\end{proposition}

\begin{proof}
It is well known that this operator together with this boundary conditions defines an essentially self-adjoint operator
(see, e.g., \cite{AIM05, BGP08, Grubb68}). We show next that its spectrum is semi-bounded below.
Its closure is a self-adjoint extension of the Laplace operator defined on $\H_0^2[0,2\pi]$.
The latter operator has finite dimensional deficiency indices and its Dirichlet extension is known to have
empty essential spectrum. According to \cite[Theorem 8.18]{We80} all the self-adjoint extensions of a
closed, symmetric operator with finite deficiency indices have the same essential spectrum and therefore the spectrum of $R$ is discrete.

Consider now the following spectral problem:
\begin{equation}
R\Phi=\Lambda\Phi,\quad \frac{\partial \Phi}{\partial r}\Bigr|_{r=0}=0,\quad \frac{\partial \Phi}{\partial r}\Bigr|_{r=2\pi}=c\Phi|_{r=2\pi}\;,
\end{equation}
with $c$ a real constant.
On general solutions $\Phi(r)=Ae^{\mathbf{i}\lambda r}+Be^{-\mathbf{i}\lambda r}$ we impose the boundary conditions.
For nonzero solutions we obtain the following relation
\begin{equation}\label{spectral function}
	-(\mathbf{i}\lambda+c)e^{-\mathbf{i}2\pi\lambda}+(\mathbf{i}\lambda-c)e^{\mathbf{i}2\pi\lambda}=0\;,
\end{equation}
where $\Lambda=\lambda^2\in\mathbb{R}$. The equation is symmetric under the interchange $\lambda\to-\lambda$.
It is therefore enough to consider either $\lambda\geq 0$ or $\lambda=\mathbf{i}\mu$ with $\mu>0$.
These two choices correspond to the positive and negative eigenvalues, respectively.
The imaginary part of Eq.~\eqref{spectral function} vanishes identically.
If $\lambda\geq 0$ its real
part takes the form $$\tan 2\pi\lambda= -\frac{c}{\lambda}\;,$$ which leads to infinite solutions for each
$c\in\mathbb{R}$ and therefore there are infinite positive eigenvalues.
If $\lambda=\mathbf{i}\mu$ we obtain from Eq.~\eqref{spectral function}
$$e^{-4\pi\mu}=\frac{\mu-c}{\mu+c}\;,$$
which has either no solution for $c<0$, the trivial solution $\mu=0$ for $c=0$
and exactly one negative solution for $c>0$. So the operator $R$ is positive for $c\leq0$ and semi-bounded below for $c>0$.
We denote the lowest possible eigenvalue by $\Lambda_0$.
\end{proof}

\begin{definition}\label{Deftensorproduct}
Consider the interval $I=[0,2\pi]$ and let $\{\Gamma_i(\bt)\}\subset\H^0(\pO)$ be an orthonormal basis.
Consider the following operator $A$ on the tensor product
$\H^0(I)\otimes\H^0(\pO)\simeq \H^0(I\times\pO)$
given by
$$A\colon\D(A)\to \H^0(I)\otimes\H^0(\pO) \quad\text{where}\quad A:=R\otimes\mathbb{I}\;,$$
on its natural domain
$$\D(A)=\Bigl\{\Phi\in\H^0(I)\otimes\H^0(\pO)\;\bigr|\;\; \Phi=\sum_{i=1}^n\Phi_i(r)\Gamma_i(\bt)\,,\;n\in\mathbb{N}\,,\;
               \Phi_i\in\D(R) \Bigr\}\;.$$
\end{definition}

\begin{proposition}\label{Asemibounded}
The operator $A$ is essentially self-adjoint, semi-bounded below and has the same lower bound $\Lambda_0$
as the operator $R$ of Proposition~\ref{prop: intervalrobin}.
\end{proposition}

\begin{proof}
Let $\Psi\in\ker(A^\dagger\mp\mathbf{i})$ and consider its decompostion in terms of the orthonormal basis
$\{\Gamma_i(\theta)\}\subset\H^0(\pO)$ such that
$\Psi=\sum_{i=0}^\infty\Psi(r)_i\Gamma_i(\bt)$. We have that $\scalar{\Psi}{(A\pm\mathbf{i})\Phi}=0\quad\forall \Phi\in\D(A)\;.$
In particular for any $\Phi=\Phi_{i_0}\Gamma_{i_0}\in\D(A)$. Then
\begin{align*}
0=\scalar{\Psi}{(A\pm\mathbf{i})\Phi_{i_0}\Gamma_{i_0}}&=\sum_i^{\infty}\scalar{\Psi_i}{(R\pm\mathbf{i})
  \Phi_{i_0}}_{\H^0(I)}\scalar{\Gamma_i}{\Gamma_{i_0}}_{\H^{0}(\pO)}\\
&=\scalar{\Psi_{i_0}}{(R\pm\mathbf{i})\Phi_{i_0}}_{\H^0(I)}\quad\forall \Phi_{i_0}\in\D(R)\;.
\end{align*}
This implies that $\Psi_{i_0}=0$ because, by Proposition~\ref{prop: intervalrobin},
$R$ is essentially self-adjoint. Therefore
$\Psi=0$ and $A$ is essentially self-adjoint.

Finally we show the semi-boundedness condition.
Using the orthonormality of the basis $\{\Gamma_i(\bt)\}$
and for any $\Phi\in\D(A)$
we have that
$$\scalar{\Phi}{A\Phi}_{\H^0(I\times\pO)}=\sum_{i=1}^n\scalar{\Phi_i}{R\Phi_i}_{\H^0(I)}\geq\Lambda_0\sum_{i=1}^n\scalar{\Phi_i}{\Phi_i}_{\H^0(I)}
  =\Lambda_0\scalar{\Phi}{\Phi}_{\H^0(I\times\pO)}\;.$$
\end{proof}

\subsection{Quadratic forms and extensions of the minimal Laplacian}

We begin associating quadratic forms to some of the operators on a collar neighborhood of the precedent subsection.

\begin{lemma}\label{AdomainH1}
Denote by $Q_A$ the closed quadratic form represented by the closure of $A$. Then its domain $\D(Q_A)$
contains the Sobolev space of class 1. For any $\Phi\in\H^1(I\times\pO)\subset\D(Q_A)$ we have the expression
$$Q_A(\Phi)=
\int_{\pO}\Bigl[\int_I \frac{\partial \bar{\Phi}}{\partial r}\frac{\partial \Phi}{\partial r}\d r- c|\gamma(\Phi)|^2 \Bigr]\d\mu_{\partial \eta}
\;.$$
\end{lemma}

\begin{proof}
Let $\Phi\in\D(A)$. Then we have recalling the boundary conditions specified in the domain $\D(R)$ that
\begin{align}
Q_A(\Phi)&=\scalar{\Phi}{A\Phi}_{\H^0(I\times\pO)}=\sum_i\scalar{\Phi_i}{R\Phi_i}_{\H^0(I)}\notag\\
&=\sum_i\scalar{\frac{\partial \Phi_i}{\partial r}}{\frac{\partial \Phi_i}{\partial r}}_{\H^0(I)}-c\bar{\Phi}_i(0)\Phi_i(0)\notag\\
&=\int_{\pO}\Bigl[\int_I \frac{\partial \bar{\Phi}}{\partial r}\frac{\partial \Phi}{\partial r}\d r -c|\varphi|^2\Bigr]\d\mu_{\partial \eta}\;.
   \label{QFA}
\end{align}
Now it is easy to check that the graph norm of this quadratic form is dominated by the Sobolev norm of order 1, $\H^1(I\times\pO)$ \,.
\begin{align*}
\normm{\Phi}^2_{Q_A}&=(1+|\Lambda_0|)\norm{\Phi}^2_{\H^0(I\times\pO)}+Q_A(\Phi)\\
&\leq(1+|\Lambda_0|)\norm{\Phi}^2_{\H^0(I\times\pO)}+\int_{\pO}\int_I
     \frac{\partial \bar{\Phi}}{\partial r}\frac{\partial \Phi}{\partial r}\d r\d\mu_{\partial \eta} +c\norm{\varphi}^2_{\H^0(\pO)}\\
&\leq (1+|\Lambda_0|)\norm{\Phi}^2_{\H^0(I\times\pO)}+C\norm{\Phi}^2_{\H^1(I\times\pO)}\\
&\leq C'\norm{\Phi}^2_{\H^1(I\times\pO)}\;,
\end{align*}
where in the second step we have used again the equivalence appearing in Proposition \ref{equivalentsobolev} and Theorem \ref{LMtracetheorem}.
The above inequality shows that $\overline{\D(A)}^{\norm{\cdot}_1}\subset \D(Q_A)$. Moreover, Corollary~\ref{Corsemibounded}
states that $\D(A)$ is dense in $\H^1(I\times\pO)$. Hence the expression Eq.~\eqref{QFA} holds also on $\H^1(I\times\pO)$.
\end{proof}

\begin{theorem}\label{maintheorem1}
Let $U\colon\H^0(\pO)\to\H^0(\pO)$ be a unitary operator with gap at $-1$.
Then the quadratic form $Q_U$ of Definition \ref{DefQU} is semi-bounded below.
\end{theorem}

\begin{proof}
Let $(\Omega,\pO,\eta)$ be a compact, Riemannian manifold with boundary. One can always select a collar neighborhood $\Xi$ of the boundary with coordinates $(r,\bt)$ such that $\Xi\simeq[-L,0]\times\pO$ and where $$\eta(r,\bt)=\begin{bmatrix}1 & 0 \\ 0 & g(r,\bt) \end{bmatrix}\;.$$ The normal vector field to the boundary is going to be $\frac{\partial}{\partial r}$\,. With this choice the induced Riemannian metric at the boundary becomes $\partial\eta(\bt)\equiv g(0,\bt)$\,. The thickness $L$ of the collar neighborhood $\Xi$ can be also selected such that it exists $\delta\ll1$ that verifies
\begin{equation}\label{compactcollar}
	(1-\delta)\sqrt{|g(0,\bt)|}\leq \sqrt{|g(r,\bt)|}\leq (1+\delta)\sqrt{|g(0,\bt)|}\;.
\end{equation}
The quadratic form $Q_U$ can be adapted to this splitting. Let $\Phi\in\D_U\subset\H^1(\Omega)$. Obviously $\Phi|_\Xi\in\H^1(\Xi)\simeq\H^1(I\times\pO)$\,. In what follows, to simplify the notation and since there is no risk of confusion, the symbol $\Phi$ will stand for both $\Phi\in\H^1(\Omega)$ and $\Phi|_{\Xi}\in\H^1(\Xi)$.
\begin{subequations}
\begin{align}
Q_U(\Phi)&=\int_\Omega\eta^{-1}(\d\bar{\Phi},\d \Phi)\d\mu_\eta-\int_\pO\bar{\varphi}A\varphi\d\mu_{\partial \eta}\\
&=\int_\Xi\eta^{-1}(\d\bar{\Phi},\d \Phi)\d\mu_\eta+\int_{\Omega\backslash\Xi}\eta^{-1}(\d\bar{\Phi},\d \Phi)\d\mu_\eta
  -\int_\pO\bar{\varphi}A\varphi\d\mu_{\partial \eta}\\
&\geq \int_\Xi\eta^{-1}(\d\bar{\Phi},\d \Phi)\d\mu_\eta-\int_\pO\bar{\varphi}A\varphi\d\mu_{\partial \eta}\label{semibounded1}\\
&=\int_\pO\int_I\Bigl[\frac{\partial \bar{\Phi}}{\partial r}\frac{\partial \Phi}{\partial r}+ g^{-1}(\d_{\bt}\Phi,\d_{\bt}\Phi)\Bigr]
  \sqrt{|g(r,\bt)|}\d r\wedge\d\bt-\int_\pO\bar{\varphi}A\varphi\d\mu_{\partial \eta}\\
&\geq \int_\pO\int_I\frac{\partial \bar{\Phi}}{\partial r}\frac{\partial \Phi}{\partial r}\sqrt{|g(r,\bt)|}\d r\wedge\d\bt
  -\int_\pO\bar{\varphi}A\varphi\d\mu_{\partial \eta}\label{semibounded2}\\
&\geq (1-\delta)\int_\pO\int_I\frac{\partial \bar{\Phi}}{\partial r}\frac{\partial \Phi}{\partial r}
  \sqrt{|g(0,\bt)|}\d r\wedge\d\bt-\int_\pO\bar{\varphi}A\varphi\sqrt{|g(0,\bt)|}\d\bt\label{semibounded3}\\
&\geq(1-\delta)\int_\pO\Bigl[\int_I\frac{\partial \bar{\Phi}}{\partial r}\frac{\partial \Phi}{\partial r}\d r
  -\frac{\norm{A}}{(1-\delta)}|\varphi|^2\Bigr]\sqrt{|g(0,\bt)|}\d\bt\\
&\geq -|\Lambda_0|(1-\delta)\norm{\Phi}^2_{\H^0(I\times\pO)}\geq -|\Lambda_0|\frac{1-\delta}{1+\delta}\norm{\Phi}^2_{\H^0(\Xi)}\geq
  -|\Lambda_0|\frac{1-\delta}{1+\delta}\norm{\Phi}^2_{\H^0(\Omega)}\;.\label{semibounded4}
\end{align}\\
\end{subequations}	
In the step leading to \eqref{semibounded1} we have used the fact that the second term is positive. In the step leading to
\eqref{semibounded2} we have used that the second term in the first integrand is positive. Then \eqref{semibounded3} follows using
the bounds \eqref{compactcollar}. The last chain of inequalities follows by Proposition~\ref{Asemibounded} and Lemma~\ref{AdomainH1},
taking $c=\norm{A}/(1-\delta)$\,.
Notice that the semi-bound of Proposition \ref{Asemibounded} is always negative in this case because $c=\norm{A}/(1-\delta)>0$.
In Definition \ref{Defunidimensional} the interval $I$ was taken of length $2\pi$ whereas in this case it has length $L$.
This affects only in a constant factor that can be absorbed in the constant $c$ by means of a linear transformation of the
manifold $T\colon [0,2\pi]\to I$\,.
\end{proof}

\begin{theorem}\label{maintheorem2}
Let $U:\H^0(\pO)\to\H^0(\pO)$ be an admissible, unitary operator. Then the quadratic form $Q_U$ of Definition \ref{DefQU} is closable.
\end{theorem}

\begin{proof}
According to Remark \ref{Remclosable} a quadratic form is closable iff for any $\Phi\in\overline{\D_U}^{\normm{\cdot}_{Q_U}}$
such that the corresponding Cauchy sequence $\{\Phi_n\}$ verifies $\norm{\Phi_n}\to0$ then
$Q(\Phi)=0$. Let $\Phi\in\overline{\D_U}^{\normm{\cdot}_{Q_U}}$. \\

(a) Lets show that it exist $\{\tilde{\Phi}_n\}\in\C^\infty(\Omega)$ such that $\normm{\Phi-\tilde{\Phi}_n}_{Q_U}\to 0$ and $\norm{\dot{\tilde{\varphi}}_n-A_U\tilde{\varphi}_n}_{\H^{1/2}(\pO)}\to0$\,. It exists $\{\Phi_n\}\in\D_U\subset\H^1(\Omega)$ such that $\normm{\Phi-\Phi_n}_{Q_U}\to 0$\,. For the sequence $\{\Phi_n\}$ take $\{\tilde{\Phi}_n\}\in\C^{\infty}(\Omega)$ as in Corollary \ref{Corclosable}. Then we have that
\begin{align*}
\normm{\Phi-\tilde{\Phi}_n}_{Q_U}&\leq \normm{\Phi-\Phi_n}_{Q_U}+\normm{\Phi_n-\tilde{\Phi}_n}_{Q_U}\\
&\leq \normm{\Phi-\Phi_n}_{Q_U}+K\norm{\Phi_n-\tilde{\Phi}_n}_1\;,
\end{align*}
where we have used Proposition \ref{H1bound}.\\

(b) Lets assume that $\norm{\Phi_n}\to0$. This implies that $\norm{\tilde{\Phi}_n}\to0$.
For every $\Psi\in\H^2_0=\D(\Delta_{\mathrm{min}})$ we have that
$$|\scalar{\Delta_{\mathrm{min}}\Psi}{\tilde{\Phi}_n}|\leq \norm{\Delta_{\mathrm{min}}\Psi}\norm{\tilde{\Phi}_n}\to 0\;.$$
Hence $\lim\tilde{\Phi}_n\in\D(\Delta_{\mathrm{min}}^\dagger)=\D(\Delta_{\mathrm{max}})$. According to
Theorem~\ref{weaktracetheorem} the traces of such functions exist and are elements of $\H^{-1/2}(\pO)$,
i.e., $\tilde{\varphi}_n\stackrel{\H^{-1/2}(\pO)}{\to}\tilde{\varphi}$\,.\\

(c) Finally we have that
\begin{align*}
Q_U(\Phi)&=\lim_{m\to\infty}\lim_{n\to\infty}\left[\scalar{\d\tilde{\Phi}_n}{\d\tilde{\Phi}_m}
           -\scalarb{\tilde{\varphi}_n}{A_U\tilde{\varphi}_m}\right]\\
&=\lim_{m\to\infty}\lim_{n\to\infty}\left[\scalar{\tilde{\Phi}_n}{-\Delta_\eta\tilde{\Phi}_m}+\scalarb{\tilde{\varphi}_n}{\dot{\tilde{\varphi}}_m}
   -\scalarb{\tilde{\varphi}_n}{A_U\tilde{\varphi}_m}\right]\\
&=\lim_{m\to\infty}\pairb{\tilde{\varphi}}{\dot{\tilde{\varphi}}_m-A_U\tilde{\varphi}_m}=0\;.
\end{align*}
Notice that in the last step we have used the continuous extension given in Proposition \ref{proppairing} of the scalar
product of the boundary $\scalarb{\cdot}{\cdot}$ to the pairing $\pairb{\cdot}{\cdot}:\H^{-1/2}(\pO)\times\H^{1/2}(\pO)\to\mathbb{C}$
associated to the scale of Hilbert spaces $\H^{1/2}(\pO)\subset\H^0(\pO)\subset\H^{-1/2}(\pO)$\,.
\end{proof}

Theorem \ref{maintheorem1} and Theorem \ref{maintheorem2} ensure that Theorem \ref{fundteo} applies and that the closure of the quadratic
form $Q_U$ for an admissible unitary $U$ is representable by means of a unique self-adjoint operator $T$,
with domain $\D(T)\subset\D(\overline{Q}_U):=\overline{\D_U}^{\normm{\cdot}_{Q_U}}$, i.e.,
$$\overline{Q}_U(\Psi,\Phi)=\scalar{\Psi}{T\Phi}\quad\Psi\in\D(\overline{Q}_U),\Phi\in\D(T)\;.$$
The following theorem establishes the relation between this operator $T$ and the Laplace-Beltrami operator.

\begin{theorem}\label{DeltaUextDeltamin}
Let $T$ be the self-adjoint operator with domain $\D(T)$ representing the closed quadratic form $\overline{Q}_U$ with domain $\D(\overline{Q}_U)$. The operator $T$ is a self-adjoint extension of the closed symmetric operator $-\Delta_{\mathrm{min}}$.
\end{theorem}

\begin{proof}
By Theorem \ref{fundteo} we have that $\Phi\in\D(T)$ iff $\Phi\in\D(\overline{Q}_U)$ and it exists $\chi\in\H^0(\Omega)$
such that $$Q_U(\Psi,\Phi)=\scalar{\Psi}{\chi}\quad\forall\Psi \in \D(\overline{Q}_U)\;.$$
Let $\Phi\in\H^2_0(\Omega)\subset\D_U$ and $\Psi\in\D_U$. Then
\begin{align*}
Q(\Psi,\Phi)&=\scalar{\d\Psi}{\d\Phi}-\scalarb{\psi}{A\varphi}\\
&=\scalar{\Psi}{-\Delta_{\mathrm{min}}\Phi}+\scalarb{\psi}{\dot{\varphi}}-\scalarb{\psi}{A\varphi}\\
&=\scalar{\Psi}{-\Delta_{\mathrm{min}}\Phi}\;.
\end{align*}
Since $\D_U$ is a core for $\overline{Q}_U$ and $\D(\overline{Q}_U)\subset\H^0(\Omega)$
the above equality holds also for every $\Psi\in\D(\overline{Q}_U)$.
Therefore $\D(\Delta_{\mathrm{min}})=\H^2_0(\Omega)\subset\D(T)$ and moreover $T|_{\D(\Delta_{\mathrm{min}})}=-\Delta_{\mathrm{min}}$.
\end{proof}

\section{Examples}\label{sec:examples}

In this section we introduce some examples that show that the characterization of the quadratic forms of Section \ref{sec:class}
and Section \ref{sec:closable and semibounded qf} include a large class of possible self-adjoint extensions of the Laplace-Beltrami operator.
This section also illustrates the simplicity in the description of extensions using admissible unitaries at the boundary.

As the boundary manifold $\pO$ is an $(n-1)$-dimensional, smooth manifold, there always exist a
$(n-1)$-simplicial complex  $\mathcal{K}$ and a smooth diffeomorphism $f:\mathcal{K}\to\pO$ such that $f(\mathcal{K})=\pO$, cf., \cite{Wh40,Wh57}. Any simplex in the complex is
diffeomorphic to a reference polyhedron $\Gamma_0\subset\mathbb{R}^{n-1}$.
The simplicial complex $\mathcal{K}$ defines therefore a triangulation of the boundary $\pO=\cup_{i=1}^N\Gamma_i$, where $\Gamma_i:=f(A_i)$,
$A_i\in\mathcal{K}$. For each element of the triangulation $\Gamma_i$ it exists a diffeomorphism $g_i:\Gamma_0\to\Gamma_i$.
Consider a reference Hilbert space $\H^0(\Gamma_0,\d\mu_0)$, where $d\mu_0$ is a fixed smooth volume element.
Each diffeomorphism $g_i$ defines a unitary transformation as follows:
\begin{definition}\label{def unitary transformation}
Let $|J_i|$ be the Jacobian determinant of the transformation of coordinates given by the diffeomorphism $g_i:\Gamma_0\to\Gamma_i$\,.
Let $\mu_i\in\C^\infty(\pO)$ be the proportionality factor $g_i^{\star}\d\mu_{\partial\eta}=\mu_i\d\mu_0$\,,
where $g_i^{\star}$ stands for the pull-back of the diffeomorphism. The unitary transformation
$T_i:\H^0(\Gamma_i,\d\mu_{\partial\eta})\to\H^0(\Gamma_0,\d\mu_0)$ is defined by
\begin{equation}\label{unitary transformation}
T_i\Phi:=\sqrt{|J_i|\mu_i}(\Phi\circ g_i)\;.
\end{equation}
\end{definition}
We show that the transformation above is unitary. First note that $T$ is invertible. It remains to
show that $T$ is an isometry:
\begin{align*}
\scalar{\Phi}{\Psi}_{\Gamma_i}&=\int_{\Gamma_i}\overline{\Phi}\Psi\d\mu_{\partial \eta}\\
&=\int_{\Gamma_0}(\overline{\Phi\circ g_i})(\Psi\circ g_i)|J_i|g_i^{\star}\d\mu_{\partial \eta}\\
&=\int_{\Gamma_0}(\overline{\Phi\circ g_i})(\Psi\circ g_i)|J_i|\mu_i \d\mu_0=\scalar{T_i\Phi}{T_i\Psi}_{\Gamma_0}\;.
\end{align*}

\begin{example}\label{periodic}
Consider that the boundary of the Riemannian manifold $(\Omega,\pO,\eta)$ admits a triangulation of two elements,
i.e., $\pO=\Gamma_1\cup\Gamma_2$\,. The Hilbert space of the boundary satisfies
$\H^0(\pO)=\H(\Gamma_1\cup\Gamma_2)\simeq \H^0(\Gamma_1)\oplus\H^0(\Gamma_2)$.
The isomorphism is given explicitly by the characteristic functions $\chi_i$ of the submanifolds $\Gamma_i$, $i=1,2$\,.
Modulo a null measure set we have that $$\Phi=\chi_1\Phi+\chi_2\Phi\;.$$ We shall define unitary operators $U=\H^0(\pO)\to\H^0(\pO)$
that are adapted to the block structure induced by the latter direct sum:
$$U=\begin{bmatrix} U_{11} & U_{12} \\ U_{21} & U_{22} \end{bmatrix}\;,$$
where $U_{ij}:\H^0(\Gamma_{j})\to\H^0(\Gamma_{i})$\,. Hence consider the following unitary operator
\begin{equation}\label{Eqperiodic}
U=\begin{bmatrix} 0 & T_1^*T_2 \\ T_2^*T_1 & 0 \end{bmatrix}\;,
\end{equation}
where the unitaries $T_i$ are defined as in Definition \ref{def unitary transformation}. Clearly, $U^2=\mathbb{I}$, and therefore the spectrum of $U$ is $\sigma(U)=\{-1,1\}$ with the corresponding orthogonal projectors given by $$P^{\bot}=\frac{1}{2}(\mathbb{I}-U)\;,$$ $$P=\frac{1}{2}(\mathbb{I}+U)\;.$$ The partial Cayley transform $A_U$ is in this case the null operator, since $P(\mathbb{I}-U)=0$. The unitary operator is therefore admissible and the corresponding quadratic form will be closable. The domain of the corresponding quadratic form $Q_U$ is given by all the functions $\Phi\in\H^1(\Omega)$ such that $P^\bot\gamma(\Phi)=0$, which in this case becomes
\begin{equation}
P^\bot\gamma(\Phi)=\frac{1}{2}
\begin{bmatrix}
\mathbb{I}_1 & -T_1^*T_2 \\ -T_2^*T_1 & \mathbb{I}_2
\end{bmatrix}\begin{bmatrix} \chi_1\gamma(\Phi) \\ \chi_2 \gamma(\Phi) \end{bmatrix}=
\begin{bmatrix}
\chi_1\gamma(\Phi)-T_1^*T_2\chi_2\gamma(\Phi)\\
-T_2^*T_1\chi_1\gamma(\Phi)+\chi_2\gamma(\Phi)
\end{bmatrix}=0\;.
\end{equation}
We can rewrite the last condition as
\begin{equation}
T_1(\chi_1\gamma(\Phi))=T_2(\chi_2\gamma(\Phi))\;.
\end{equation}	
More concretely, this boundary conditions describe generalized periodic boundary conditions identifying the two
triangulation elements of the boundary with each other. The unitary transformations $T_i$ are necessary to make the triangulation
elements congruent. In particular,
if $(\Gamma_1,\eta_1)$ and $(\Gamma_2,\eta_2)$ are isomorphic as Riemannian manifolds then one can
recover the standard periodic boundary conditions.
\end{example}

\begin{example}\label{ex:quasiperiodic}
Consider the same situation as in the previous example but with the unitary operator replaced by
\begin{equation}\label{Eqquasiperiodic}
U=\begin{bmatrix} 0 & T_1^*e^{i\alpha}T_2 \\ T_2^*e^{-i\alpha}T_1 & 0 \end{bmatrix}\,,\quad \alpha\in \C^\infty(\Gamma_0)\;.
\end{equation}
In this case we have also that $U^2=\mathbb{I}$ and the calculations of the previous example can be applied step by step.
More concretely $P^\bot=(I-U)/2$ and the partial Cayley transform also vanishes. The boundary condition becomes in this case
\begin{equation}\label{quasiperiodic}
T_1(\chi_1\gamma(\Phi))=e^{i\alpha}T_2(\chi_2\gamma(\Phi))\;.
\end{equation}
This boundary conditions can be called generalized, quasiperiodic boundary conditions. For simple geometries and constant function
$\alpha$ these are the boundary conditions that define the periodic Bloch functions.
\end{example}

The condition $\alpha\in \C^\infty(\Gamma_0)$ in the example above can be relaxed. First we will show that the isometries $T_i$
do preserve the regularity of the function.

\begin{proposition}\label{prop regularity}
Let $T_i$ be a unitary transformation as given by Definition \ref{def unitary transformation}. Let $\Phi\in\H^k(\Gamma_i)$, $k\geq0$.
Then $T_i\Phi\in\H^k(\Gamma_0)$.
\end{proposition}

\begin{proof}
It is well known, cf. \cite[Theorem 3.41]{Ad03} or \cite[Lemma 7.1.4]{davies:95}, that the pull-back of a function under a smooth diffeomorphism
$g:\Omega_1\to\Omega_2$ preserves the regularity of the function,
i.e., $g^{\star}\Phi\in\H^k(\Omega_1)$ if $\Phi\in\H^k(\Omega_2)$, $k\geq0$.
It is therefore enough to prove that multiplication by a smooth positive function also preserves the regularity.
According to Definition \ref{DefSobolev2} it is enough to prove it for a smooth, compact, boundaryless Riemannian manifold
$(\tilde{\Omega},\tilde{\eta})$ and to consider that $\Phi\in\C^\infty(\tilde{\Omega})$, since this set is dense in $\H^k(\tilde{\Omega})$.
Let $f\in\C^\infty(\tilde{\Omega})$\,.
\begin{align*}
\int_{\tilde{\Omega}} \overline{f\Phi}(I-\Delta_{\tilde{\eta}})^k(f\Phi)\d\mu_{\tilde{\eta}}
&\leq \sup_{\tilde{\Omega}}|f|\int_{\tilde{\Omega}}\overline{\Phi}(I-\Delta_{\tilde{\eta}})^k(f\Phi)\d\mu_{\tilde{\eta}}\\
&\leq \sup_{\tilde{\Omega}}|f|\int_{\tilde{\Omega}}\overline{(I-\Delta_{\tilde{\eta}})^k\Phi}f\Phi\d\mu_{\tilde{\eta}}\\
&\leq (\sup_{\tilde{\Omega}}|f|)^2\int_{\tilde{\Omega}}\overline{(I-\Delta_{\tilde{\eta}})^k\Phi}\Phi\d\mu_{\tilde{\eta}}<\infty\;.
\end{align*}
We have used Definition \ref{DefSobolev} directly and the fact that the operator $(I-\Delta_{\tilde{\eta}})^k$ is
essentially self-adjoint over the smooth functions.
\end{proof}

According to Proposition~\ref{prop regularity} we have that $T_i(\chi_i\gamma(\Phi))\in\H^{1/2}(\Gamma_0)$, $i=1,2$.
Therefore, to get nontrivial solutions for the expression \eqref{quasiperiodic}, the function $\alpha:\Gamma_0\to[0,2\pi]$
can be chosen such that $e^{i\alpha}T_2(\chi_2\gamma)\in\H^{1/2}(\Gamma_0)$. Since $\C^0(\Gamma_0)$ is a dense subset in $\H^{1/2}(\Gamma_0)$,
and pointwise multiplication is a continuous operation for continuous functions it is enough to consider $\alpha\in\C^0(\Gamma_0)$.

\begin{example}\label{generalized Robin}
Consider that the boundary of the Riemannian manifold $(\Omega,\pO,\eta)$ admits a triangulation of two elements like in the Example \ref{periodic}.
So we have that $\pO=\Gamma_1\cup\Gamma_2$\,. Consider the following unitary operator $U:\H^0(\pO)\to\H^0(\pO)$ adapted to the block structure
defined by this triangulation
\begin{equation}\label{EqRobin}
U=\begin{bmatrix} e^{\mathrm{i}\beta_1}\mathbb{I}_1 & 0\\ 0 & e^{\mathrm{i}\beta_2}\mathbb{I}_2 \end{bmatrix}\;,
\end{equation}
where $\C^0(\Gamma_i)\ni\beta_i:\Gamma_i\to [-\pi+\delta,\pi-\delta]$ with $\delta>0$. The latter condition guaranties that the unitary matrix
has gap at $-1$. Since the unitary is diagonal in the block structure, it is clear that $P^\bot=0$\,.
The domain of the quadratic form $Q_U$ is given in this case by all the functions $\Phi\in\H^1(\Omega)$ \,. The partial Cayley transform is
in this case the operator $A_U=\H^0(\pO)\to\H^0(\pO)$ defined by
\begin{equation}\label{partial Cayley Robin}
A_U=\begin{bmatrix} -\tan\frac{\beta_1}{2} & 0 \\ 0 & -\tan{\frac{\beta_2}{2}} \end{bmatrix}\;.
\end{equation}
A matrix like the one above will lead to self-adjoint extensions of the Laplace-Beltrami operator that verify generalized Robin type boundary
conditions $\chi_i\dot{\varphi}=-\tan\frac{\beta_i}{2}\chi_i\varphi$. Unfortunately, the partial Cayley transform does not satisfy the
admissibility condition in this case. Nevertheless, we will show that the quadratic form above is indeed closable.
\end{example}

Given a triangulation of the boundary $\pO=\cup_{i=1}^N\Gamma_i$ we can consider the Hilbert space that results of the direct sum of the corresponding Sobolev spaces. We will denote it as $$\oplus\H^k:=\oplus_{i=1}^{N}\H^{k}(\Gamma_i)\;.$$ Assuming that the partial Cayley transform verifies the condition $$\norm{A_U\gamma(\Phi)}_{\oplus\H^{1/2}}\leq K \norm{\gamma(\Phi)}_{\oplus\H^{1/2}}\;,$$ we can generalize Lemma \ref{Lemma approxdotphi} and Corollary \ref{Corclosable} as follows.

\begin{lemma}[Lemma $\ref{Lemma approxdotphi}^*$]\label{Lemma bis}
	Let $\Phi\in\H^1(\Omega)$, $f\in\oplus\H^{1/2}$. Then, for every $\epsilon>0$ it exists $\tilde{\Phi}\in\C^\infty(\Omega)$ such that $\norm{\Phi-\tilde{\Phi}}_1<\epsilon$, $\norm{\varphi-\tilde{\varphi}}_{\H^{1/2}(\pO)}<\epsilon$ and $\norm{f-\dot{\tilde{\varphi}}}_{\oplus\H^{1/2}}<\epsilon$\,.
\end{lemma}

\begin{proof}
 The proof of this lemma follows exactly the one for the Lemma \ref{Lemma approxdotphi}. It is enough to notice that the space $\H^1(\pO)$ is dense in $\oplus\H^{1/2}$.
\end{proof}

\begin{corollary}[Corollary $\ref{Corclosable}^*$]\label{Corollary bis}
Let $\{\Phi_n\}\subset\H^1(\Omega)$ and let $A_U$ be the partial Cayley transform of a unitary operator with gap at $-1$
such that $\norm{A\gamma(\Phi)}_{\oplus\H^{1/2}}\leq K \norm{\gamma(\Phi)}_{\oplus\H^{1/2}}$\,.
Then it exists a sequence of smooth functions $\{\tilde{\Phi}_n\}\in\C^{\infty}(\Omega)$ such that $\norm{\Phi_n-\tilde{\Phi}_n}_{\H^1(\Omega)}<\frac{1}{n}$\,,
$\norm{\varphi_n-\tilde{\varphi}_n}_{\H^{1/2}(\pO)}<\frac{1}{n}$\,, and
$\norm{\dot{\tilde{\varphi}}_n-A_U\tilde{\varphi}_n}_{\oplus\H^{1/2}}<\frac{1}{n}$\,.
\end{corollary}

\begin{proof}
The proof is the same as for Corollary \ref{Corclosable} but now we take $\tilde{\Phi}_{n_0}$ as in Lemma \ref{Lemma bis} with
$f=A_U\varphi_{n_0}\in\oplus\H^{1/2}$\,.
\end{proof}

Now we can show that the quadratic forms $Q_U$ defined for unitary operators of the form appearing in Example~\ref{generalized Robin} are closable.
We show first that the partial Cayley transform of Equation \eqref{partial Cayley Robin} verifies the conditions of the
Corollary~\ref{Corollary bis} above. We have that
\begin{align*}
\norm{A_U\varphi}_{\oplus\H^{1/2}}^2&=\norm{A_U\chi_1\varphi}^2_{\H^{1/2}(\Gamma_1)}+\norm{A_U\chi_2\varphi}^2_{\H^{1/2}(\Gamma_2)}\\
&=\norm{\tan{\frac{\beta_1}{2}}\chi_1\varphi}^2_{\H^{1/2}(\Gamma_1)}+\norm{\tan{\frac{\beta_2}{2}}\chi_2\varphi}^2_{\H^{1/2}(\Gamma_2)}\\
&\leq K \left[\norm{\chi_1\varphi}^2_{\H^{1/2}(\Gamma_1)}+\norm{\chi_2\varphi}^2_{\H^{1/2}(\Gamma_2)}\right]=K\norm{\varphi}^2_{\oplus\H^{1/2}}\;.
\end{align*}
The last inequality follows from the discussion after Example \ref{ex:quasiperiodic} because the functions
$\beta_i\colon\Gamma_i\to[-\pi+\delta,\pi-\delta]$ are continuous. Take the sequence $\{\Phi_n\}\in\D_U$ as in the
proof of Theorem \ref{maintheorem2} and accordingly take $\{\tilde{\Phi}_n\}\in\C^{\infty}(\Omega)$ as in Corollary~\ref{Corollary bis}. Then we have that
\begin{align*}
|Q(\Phi)|&=\lim_{m\to\infty}\lim_{n\to\infty}\left|\scalar{\d\tilde{\Phi}_n}{\d\tilde{\Phi}_m}-\scalarb{\tilde{\varphi}_n}{A\tilde{\varphi}_m}\right|\\
&\leq\lim_{m\to\infty}\lim_{n\to\infty}\left[|\scalar{\tilde{\Phi}_n}{-\Delta_\eta\tilde{\Phi}_m}|+|\scalarb{\tilde{\varphi}_n}{\dot{\tilde{\varphi}}_m-A_U\tilde{\varphi}_m}|\right]\\
&=\lim_{m\to\infty}\lim_{n\to\infty}\left[|\scalar{\tilde{\Phi}_n}{-\Delta_\eta\tilde{\Phi}_m}|+|\smash{\sum_{i=1}^N}\scalar{\tilde{\varphi}_n}{\dot{\tilde{\varphi}}_m-A_U\tilde{\varphi}_m}_{\Gamma_i}|\right]\\
&\leq\lim_{m\to\infty}\lim_{n\to\infty}\sum_{i=1}^N|\scalar{\tilde{\varphi}_n}{\dot{\tilde{\varphi}}_m-A_U\tilde{\varphi}_m}_{\Gamma_i}|\\
&\leq\lim_{m\to\infty}\lim_{n\to\infty}\sum_{i=1}^N\norm{\chi_i\tilde{\varphi}_n}_{\H^{-1/2}(\Gamma_i)}\norm{\chi_i\dot{\tilde{\varphi}}_m-\chi_iA_U\tilde{\varphi}_m}_{\H^{1/2}(\Gamma_i)}\\
&\leq \lim_{m\to\infty}\lim_{n\to\infty} \norm{\tilde{\varphi}_n}_{\H^{-1/2}(\pO)} \smash{\sum_{i=1}^N}\norm{\chi_i\dot{\tilde{\varphi}}_m-\chi_iA_U\tilde{\varphi}_m}_{\H^{1/2}(\Gamma_i)}=0\;.
\end{align*}
We have used Definition \ref{DefSobolev2} and the structure of the scales of Hilbert spaces
$\H^{1/2}(\Gamma_i)\subset\H^0(\Gamma_i)\subset\H^{-1/2}(\Gamma_i)$\,. Hence, the unitary operators of Example \ref{generalized Robin} are closable.
In particular, this class of closable quadratic forms defines generalized Robin type boundary conditions $\dot{\varphi}=-\tan{\frac{\beta}{2}}\varphi$
where $\beta$ is allowed to be a piecewise continuous function with discontinuities at the vertices of the triangulation.

\begin{example}
Consider a unitary operator at the boundary of the form
\begin{equation}\label{Eq mixed}
	U=\begin{bmatrix} -\mathbb{I}_1 & 0\\ 0 & e^{\mathrm{i}\beta_2}\mathbb{I}_2 \end{bmatrix}\;,
\end{equation}
with $\beta_2:\Gamma_2 \to [-\pi+\delta,\pi-\delta]$ continuous. Again we need the condition $\delta>0$
in order to guaranty that the unitary matrix $U$ has gap at $-1$. In this case it is clear that
$$P^\bot=\begin{bmatrix} \mathbb{I}_1 & 0 \\ 0 & 0 \end{bmatrix}\;,$$ and that the partial Cayley transform becomes
$$A_U=\begin{bmatrix} 0 \\ -\tan{\frac{\beta_2}{2}} \end{bmatrix}\;.$$ This partial Cayley transform verifies the weaker
admissibility condition of the previous example and therefore defines a closable quadratic form too. This one defines a
boundary condition of the mixed type where
$$\chi_1\varphi=0\,,\quad\chi_2\dot{\varphi}=-\tan{\frac{\beta_2}{2}}\chi_2\varphi\;.$$
In particular when $\beta_2=0$ this mixed type boundary condition defines the boundary conditions of the so called
\emph{Zaremba problem} with $$\chi_1\varphi=0\,,\quad\chi_2\dot{\varphi}=0\;.$$
\end{example}

\begin{example}
Let $(\Omega,\pO,\eta)$ be a smooth, compact, Riemannian manifold.
Suppose that the boundary manifold admits a triangulation $\pO=\cup_{i=1}^N\Gamma_i$. Any unitary matrix that has
blockwise the structure of any of the above examples, i.e., Equations \eqref{Eqperiodic}, \eqref{Eqquasiperiodic},
\eqref{EqRobin} or \eqref{Eq mixed} leads to a closable, semi-bounded quadratic form $Q_U$.
\end{example}

\paragraph{\bf Acknowledgements:}
It is a pleasure to thank many useful conversations with Giuseppe Marmo and Manuel Asorey on this topic.
Part of this work was done while we were visiting the Department of Mathematics at UC Berkeley.
We would like to thank the hospitality and stimulating atmosphere during this visit.

\end{document}